\documentclass[11pt]{article}
\usepackage{amsmath,amsthm,amsfonts,amssymb,bm,wasysym}
\usepackage{epsfig}
\usepackage[usenames]{color}
\usepackage{verbatim}
\usepackage{hyperref}
\usepackage{multicol}
\usepackage{comment}
\usepackage{float}
\usepackage{graphicx}
\usepackage[utf8]{inputenc}
\def \ind {\hbox{1\hskip -3pt I}}
\def\bs{\backslash}

\parskip \medskipamount
\newtheorem{theorem}{Theorem}[section]

\numberwithin{equation}{section}
\newtheorem{lemma}[theorem]{Lemma}
\newtheorem{proposition}[theorem]{Proposition}
\newtheorem{corollary}[theorem]{Corollary}

\newtheorem{remark}[theorem]{Remark}

\newtheorem{claim}[theorem]{Claim}

\numberwithin{equation}{section}

\def\N{\mathbb{N}}
\def\Z{\mathbb{Z}}

\def\R{\mathbb{R}}
\def\C{\mathbb{C}}
\def\bP{\mathbb{P}}
\def\S{\mathcal{S}}

\def\AA{\mathcal{A}}
\def\CC{\mathcal{C}}
\def\HH{\mathcal{H}}

\def\F{\mathcal{F}}

\def\B{\mathcal{B}}

\def\bP{\mathbb{P}}

\renewcommand{\phi}{\varphi}
\renewcommand{\epsilon}{\varepsilon}

\def\RR{\mathcal{R}}
\allowdisplaybreaks

\newcommand{\1}{{\text{\Large $\mathfrak 1$}}}

\renewcommand{\emptyset}{\varnothing}

\def\C{{\mathcal C}}

\newcommand{\til}{\widetilde}
\def\reff#1{(\ref{#1})}

\newcommand{\pr}[1]{\mathbb{P}\!\left(#1\right)}
\newcommand{\E}[1]{\mathbb{E}\!\left[#1\right]}

\newcommand{\prstart}[2]{\mathbb{P}_{#2}\!\left(#1\right)}
\newcommand{\prcond}[3]{\mathbb{P}_{#3}\!\left(#1\;\middle\vert\;#2\right)}

\newcommand{\norm}[1]{\left\| #1 \right\|}

\newcommand{\tn}{|\kern-.1em|\kern-0.1em|}

\newcommand{\cp}{\mathrm{Cap}}

\newcommand{\vr}[1]{\mathrm{Var}\left(#1\right)}

\newcommand{\cc}[1]{\mathrm{Cap}\left(#1\right)}

\def\Gd{G_d}

\def\bP{\mathbb{P}}

\newcommand{\tb}[1]{\textbf{\color{blue}{#1}}}

\begin{document}
\title{\bf Capacity of the range of random walk on $\Z^4$}

\author{Amine Asselah \thanks{
Universit\'e Paris-Est Cr\'eteil; amine.asselah@u-pec.fr} \and
Bruno Schapira\thanks{Aix-Marseille Universit\'e, CNRS, Centrale Marseille, I2M, UMR 7373, 13453 Marseille, France;  bruno.schapira@univ-amu.fr} \and Perla Sousi\thanks{University of Cambridge, Cambridge, UK;   p.sousi@statslab.cam.ac.uk} 
}
\date{}
\maketitle
\begin{abstract}
We study the scaling limit of the capacity of the range of
a simple random walk on the integer lattice in dimension four. 
We establish a strong law of large numbers and a central
limit theorem  with a non-gaussian limit. 
The  asymptotic behaviour is analogous to that found by Le Gall in~'86~\cite{LG86} for 
the {\it volume} of the range in dimension two. 
MSC 2010 \emph{subject classifications.} Primary 60F05, 60G50.
\end{abstract}

\section{Introduction}\label{sec:intro}
This paper is devoted to the study of the capacity of the
range of a random walk in dimension four. The point of view
we adopt is that the capacity is a {\it hitting probability}.
More precisely, the capacity of a set is proportional to the
probability a random walk {\it sent from infinity} hits the set.
Then, the capacity of the range of a random
walk is cast into a problem of {\it intersection of paths}, and dimension
four is {\it critical} in view of classical results of Dvoretsky,
Erd\"os and Kakutani \cite{DEK} establishing that the paths of two
independent Brownian motions do not intersect if, and only if, 
dimension is four or larger.

The capacity of a set $A\in \Z^4$ can also be viewed as an escape probability.
Indeed, let $\bP_x$ be the law of a simple
random walk starting at $x$, let $G_d$ be the discrete Green's function,
and let $H_A$ and $H_A^+$ stand respectively for the hitting time of a finite set $A$ and the return time in $A$. Then
\begin{equation}\label{def-cap}
\cc{A}\ =\ \sum_{x\in A} \bP_x(H_A^+=\infty)\ =\ \lim_{\norm{y}\to\infty}\, 
 \frac{\bP_y(H_A<\infty)}{G_d(0,y)}.
\end{equation}
One easily passes from one representation in \reff{def-cap} to the other
using the last passage decomposition formula, see \eqref{last.passage} below.

Denote by $\{S(n),\ n\in \N\}$ a simple random walk in $\Z^4$. 
For two integers $m,n$, the range $\RR [m,n]$ (or simply $\RR_n$ when $m=0$) 
in the time period $[m,n]$ is defined as
\[
\RR[m,n]=\{S(m),\dots,S(n)\}.
\]

Our first result is a strong law of large numbers for $\cc{\RR_n}$.
 
\begin{theorem}\label{thm:lawoflargenumbers}
        Let $S$ be a simple random walk in $\Z^4$. Almost surely,
        \[
\lim_{n\to\infty}\  \frac{\log n}{n} \cdot \cc{\RR_n}\ =\  \frac{\pi^2}{8}.
        \]
\end{theorem}
Our second result is a central limit theorem for $\cc{\RR_n}$, and requires more notation:
$G$ denotes the continuous Green's function, and 
$(\beta_s, s\geq 0)$ a standard four-dimensional Brownian motion.

\begin{theorem}\label{thm:clt}
Let $S$ be a simple random walk in $\Z^4$. Then, as $n$ goes to infinity
\[
\frac{(\log n)^2}{ n}\cdot (\cc{\RR_n} - \E{\cc{\RR_n}})\quad 
\stackrel{(d)}{\Longrightarrow} \quad -\frac{\pi^4}{4} \cdot \gamma_G\big([0,1]^2\big),
\]
where $\stackrel{(d)}{\Longrightarrow}$ stands for convergence in distribution, and $\gamma_G\big([0,1]^2\big)$ is formally defined as
\begin{equation}\label{formal-gammaG}
\gamma_G\big([0,1]^2\big) \ =\ 
\int_{0}^{1}\int_0^1 G(\beta_s, \beta_t)\, ds\, dt -
 \E{\int_0^1\int_0^1 G(\beta_s,\beta_t)\,ds\, dt}.
\end{equation}
Moreover, $\gamma_G\big([0,1]^2\big)$ is non-degenerate, and non-gaussian, since for some $\lambda\in \R$,
\begin{equation}\label{infinite-moment}
\E{\exp(\lambda \gamma_G\big([0,1]^2\big))}=\infty.
\end{equation}
\end{theorem}

\begin{remark}\rm{
Although both terms appearing in the definition of
$\gamma_G([0,1]^2)$ have infinite expectation,
we make sense, in Section~\ref{sec:def}, of
$\gamma_G\big([0,1]^2\big)$ as an $L^2$-random variable following Le Gall's approach 
used to define the self-intersection local time, see \cite{LG85, LG86, LG88}. 
We also prove there 
that it has some infinite exponential moment, showing in particular that it is not Gaussian. 
}
\end{remark}

\begin{remark}
\rm{
Theorem~\ref{thm:clt} shows that the capacity of the range in $4$ dimensions behaves in the same way as the size of the range in $2$ dimensions as shown by Le Gall in '86~\cite{LG86}: 
\[
\frac{(\log n)^2}{n}\cdot 
(|\RR_n|- \E{|\RR_n|}) \quad \stackrel{(d)}{\Longrightarrow} \quad -\pi^2 \cdot \gamma([0,1]^2),
\]
where $\gamma([0,1]^2)$ is defined formally via
\[
\gamma([0,1]^2) = \int_0^1\int_0^1\delta_{(0)}(\beta_s - \beta_t)\,ds\,dt - \E{\int_0^1\int_0^1\delta_{(0)}(\beta_s - \beta_t)\,ds\,dt},
\]
where $\beta$ is a standard two-dimensional Brownian motion.}
\end{remark}

As a corollary of our result we obtain the asymptotic behaviour of the variance of~$\cp(\RR_n)$. 

\begin{corollary}\label{cor:var}
Let $S$ be a simple random walk in $\Z^4$. Then
\[
\lim_{n\to\infty}
\frac{(\log n)^4}{n^2}\vr{\cc{\RR_n}}=\frac{\pi^8}{16}\cdot \E{\left(\gamma_G\left([0,1]^2\right)\right)^2}.
\]
\end{corollary}
We note that the limiting term in this corollary is nonzero. 

Seen as a normalised hitting probability, the
capacity of the range of a walk is an object which appeared
in disguised form in topics linked with {\it intersection
of paths of random walks}. This latter topic grew already large
in the nineties, as Lawler's '91 book~\cite{Law91} testifies. One reason for
that is the many diverse sources of motivation:
(i) quantum field theories
with a seminal insight of Symanzik~\cite{Symanzik}, and with contributions by
Lawler~\cite{Law82,Law85}, Aizenman \cite{A}, 
Felder and Fr\"olich \cite{FF}, to name a few
(see the book \cite{FFS} for a historical account and references therein), 
(ii) probability and the non-Markovian model of {\it self-avoiding walk},
with contributions from Brydges and Spencer~\cite{BS}, Madras and Slade~\cite{MadrasSlade}, and Lawler~\cite{LawlerSAW},
(iii) conformal field theories, and the intersection exponents
relations in dimension two and with contributions by Duplantier and Kwon~\cite{DK}, Duplantier~\cite{duplantier1998random},  Burdzy and Lawler~\cite{BL}, and
Lawler, Schramm and Werner~\cite{LSW} and references therein.

The models studied can be either discrete random walks, or
their continuous counterpart, the Wiener sausages.
In the mid-nineties, Aizenman~\cite{A}, Albeverio and Zhou~\cite{AZ}, Pemantle, Peres and Shapiro~\cite{PPS} and Khoshnevisan~\cite{K} 
proposed useful methods and estimates for the intersection of two
Wiener sausages. These estimates were important in
understanding {\it how small} was the trace of a Brownian motion. 
In 2004, van den Berg, Bolthausen and den Hollander~\cite{BBH04} studied the
upward deviations for the volume of intersection of two Wiener
sausages, and established a Large Deviations Principle. More recently Erhard and Poisat~\cite{EP} also established large 
deviations estimates for the capacity of a Wiener sausage. 

Recently, there has been a revival of problems linked
with {\it intersection of paths}. The model
of random interlacements was invented by Sznitman~\cite{Sznitman} 
initially to study the trace of a walk, living in $d$-dimensional
torus of side $N$ for a time $N^d$.
Sznitman introduced a measure on infinite paths on the infinite lattice
of random walks whose probability of avoiding any given set is
proportional to the exponential of minus
its Newtonian capacity. Recently
Rath and Sapozhnikov~\cite{RS}, and Chang and Sapozhnikov~\cite{CS} 
established moments and deviation bounds for the capacity of the union of
ranges of paths. In~\cite{Asselah:2016hc} we observed that the precise two sided non-intersection 
bounds of Lawler \cite{Law91} yield in dimension four, 
\begin{equation}
\label{LLN1moment}
\lim_{n\to\infty} \, \frac{\log n}{n} \ \mathbb E[\cp(\RR_n)]\ =\ \frac{\pi^2}{8}. 
\end{equation}
Soon after this, 
Chang \cite{Chang} obtained a sharp estimate on the
second moment: 
        \[
\lim_{n\to\infty}\ 
\frac{\mathbb E[\cp(\RR_n)^2]}{\mathbb E[\cp(\RR_n)]^2} \ =\ 1, 
        \]
implying a weak law of large numbers. 
Chang \cite{Chang} also established a fluctuation result in dimension three by coupling the walk and the Brownian motion: 
\begin{equation}
\label{CLTd3}
\frac{\cp(\RR_n)}{\sigma \sqrt n} \quad 
\stackrel{(d)}{\Longrightarrow} \quad \cp(\beta[0,1]),
\end{equation}
with $\sigma$ some renormalising constant and $\beta[0,1]$ the trace of a three-dimensional Brownian motion between time $0$ and $1$. In \cite{Asselah:2016hc}, we also proved a standard central limit theorem in dimension larger than or equal to $6$ (with a standard $\sqrt n$ normalising factor, and a Gaussian limit), while the law of large numbers had already been obtained in dimension $5$ and larger by Jain and Orey \cite{JO69}, almost fifty years ago. 
A striking correspondence emerges: all these results for the capacity of the range are analogous 
to results for the volume of the range, see \cite{DE, JP, LG86}, 
but only after dropping space dimension by two units to go from capacity to volume of the range. 
The remaining open issue is the central limit theorem for the capacity of the range in dimension $5$.

Recently, van den Berg, Bolthausen and den Hollander \cite{BBH16} 
advocated a new geometric characteristic, the torsional
rigidity of the complement of a Wiener sausage,
as a way to probe the shape of the sausage.
In order to obtain leading asymptotics for the torsional
rigidity, one needs a law of large numbers for the capacity
of a Wiener sausage, which is not proved yet in dimension four; see however our companion paper \cite{ASS16+} 
for a partial result in this direction. Our Theorem \ref{thm:lawoflargenumbers} establishes these 
asymptotics for the discrete model, and thus prepares the study of torsional rigidity for random walk.

Our own motivation for studying the capacity of the range of
a random walk comes from studying a random
walk conditioned on being localised during a time-period $[0,N]$
in a ball of volume of order~$N$~\cite{AmineBruno}.
In this regime the localised walk
necessarily intersects often its own path, and one of the 
main technical estimates in \cite{AmineBruno} concerns
the event of visiting a set $\Lambda$ made up of non-overlapping balls of
fixed radius. We establish that visiting each ball, making up $\Lambda$,
the same number of times is related to the capacity of $\Lambda$.
This allowed us to obtain rough estimates on the capacity of the range
of a localised walk, and convinced us that the capacity of the range 
was a relevant object to consider.

\textbf{Heuristics.} We wish now to explain at a heuristic level
the scaling of the capacity in $d=4$, as well as the reason for
our Central Limit Theorem. Along the way, we present a simple
decomposition formula for the capacity of two finite sets,
and highlight the connection with the volume of the range in~$d=2$. 

Let us start by explaining in simple terms why the
scaling of the capacity of the range is $n/\log(n)$ in dimension 4.
Consider \reff{def-cap} where $A=\RR_n$, and observe that it is
enough to consider the site $x$ on the boundary of a ball, say of
radius $R$, containing the set $A$: a walker coming from infinity 
basically spreads uniformly
on the boundary of such a ball when it hits it, and \reff{def-cap} is
almost correct when considering $x$ uniformly distributed on the
boundary of ball of radius $R$ and normalised by~$\Gd(0,R)$, since
$\Gd(0,y)/\Gd(0,R)$ is the probability of eventually hitting the ball of 
radius $R$ when starting at $y$.
Now, during a time-period $[0,n]$, the walk typically stays in a ball of radius 
$R=\sqrt n$, and we consider $R$ of this order. {\it We need
therefore to estimate the probability that two independent walks
starting at a distance $\sqrt n$ meet}. More precisely, we need
to estimate 
\[
\bP_{0,x}( \RR[0,n]\cap \til{\RR}[0,\infty)\not=\emptyset),\quad
\text{with}\quad \|x\|\sim \sqrt n.
\]
To estimate this intersection event, Lawler~\cite{Law91} counts the number of times the two paths intersect.
 Its expectation is expressed
as a product of the probability the two walks meet times the mean number of meetings
after the first one, that is when they start from the same point. Then,
simple computations give the following  
\[
\mathbb{E}_{0,0}[| \RR[0,n]\cap \til{\RR}[0,\infty)|]\asymp \log n,
\quad\text{and}\quad
\mathbb{E}_{0,x}[| \RR[0,n]\cap \til{\RR}[0,\infty)|]\asymp 1,
\quad\text{when}\quad \|x\|\sim \sqrt n.
\]
Then, the order of the probability of intersection is obtained by taking the ratio
of the two previous quantities, and dividing by $\Gd(0,R)$ which is of order $1/R^2$. This is established rigorously in Section~4.3 of Lawler \cite{Law91}. The
scaling for the capacity follows, at least heuristically.

A simple and key observation of Le Gall~\cite{LG86}, using the symmetry of the increments and translation invariance
of the lattice, is that the range $\RR[0,2n]$ translated by $S(n)$, is the union of two 
independent ranges: $\RR^1_n=\RR[0,n]-S(n)$ and $\RR^2_n=\RR[n,2n]- S(n)$,
which yields by the exclusion-inclusion formula 
\begin{equation}\label{incl.excl}
|\RR^1_n\cup \RR^2_n|=|\RR^1_n|+|\RR^2_n|-|\RR^1_n\cap \RR^2_n|.
\end{equation}
This is the starting block of Le Gall's proof. 
Our starting point is that the capacity of the range 
obeys a decomposition formula which plays exactly the same role as the 
exclusion-inclusion formula does for the volume of the range.
\begin{proposition}\label{pro:deccap}
Let $A$ and $B$ be two finite subsets of $\Z^d$. We have
\begin{equation}\label{dec.intro}
\cc{A\cup B} = \cc{A} + \cc{B} - \chi(A,B) -\chi(B,A)-\epsilon(A,B),
\end{equation}
where
\[
\chi(A,B) = \sum_{y\in A} \sum_{z\in B} \prstart{H^+_{A\cup B}=
\infty}{y} \Gd(y,z)\prstart{H_B^+=\infty}{z},
\]
and $0\leq\epsilon(A,B) \leq\cc{A\cap B}$.
\end{proposition}

\begin{remark}\rm{
The equalities \reff{incl.excl} and \reff{dec.intro} explain the striking correspondence between asymptotics
for the volume and the capacity of the range. As long as the term $\epsilon$ in \reff{dec.intro} is innocuous, the order of magnitude of the 
cross term $\chi(\RR^1_n,\RR^2_n)$ 
in dimension $d+2$ is shown to be the same as the order of magnitude of the intersection term $|\RR^1_n\cap \RR^2_n|$ in dimension~$d$. }
\end{remark}

\begin{remark}\rm{
We intend to apply \reff{dec.intro}
with $A=\RR^1_n$ and $B=\RR^2_n$.
In dimension 4, it turns out that $\epsilon(A,B)$ is innocuous.
One classical inequality $\cc{A\cup B} \le \cc{A} + \cc{B} - \cc{A\cap B}$
 (see Proposition 2.3.4 of Lawler \cite{Law91})
misses in this context the $\chi$ terms in \reff{dec.intro} which dominate the fluctuations.}
\end{remark}
As in the CLT proof for the volume of the range in dimension 2  \cite{LG86},
we iterate \eqref{dec.intro} and write the capacity of the range as the sum of a rescaled
self-similar part, consisting of a sum of independent and (almost) identically distributed terms, plus a sum of cross terms. 
Our proofs establish that for the law of large numbers it is
the self-similar part which dominates, the cross terms being of smaller order than $n/\log n$, 
while for the central limit theorem, it is the opposite situation:
the fluctuations of the self-similar part are negligible compared
to those of the cross terms of order $n/\log^2 n$. 
This striking phenomenon is exactly the same as the one 
Le Gall discovered some thirty years ago, 
when proving the central limit theorem for the volume of the range \cite{LG86}. 

We are now in a position to shed some light on the form of our CLT.
We consider $\chi(\RR^1,\RR^2)$ and expect, as
Theorem~\ref{thm:lawoflargenumbers} essentially teaches,
that typically and to leading order
\[
\text{for}\quad x\in \RR^1_n,\quad
\bP_x\big(H^+_{\RR^1_n\cup \RR^2_n}=\infty\big) \ \sim\ \frac{\pi^2}{8}
\cdot \frac{1}{\log n},\quad
\text{for}\quad y\in \RR^2_n,\quad
\bP_y\big(H^+_{\RR^2_n}=\infty\big) \ \sim\ \frac{\pi^2}{8}
\cdot \frac{1}{\log n}.
\]
Note that in dimension higher than four, typically for $x\in \RR[0,n]$,
$\prstart{H^+_{\RR[0,n]}=\infty}{x}$ is rather of order 1.
Our key technical estimates is then to make the escape events into
local events (in a space scale much smaller than $\sqrt n$),
and thus transform the intersection term in a term looking
to leading order like
\begin{equation}\label{sketch-4}
\Big(\frac{\pi^2}{8}\cdot 
\frac{1}{\log n}\Big)^2\cdot 
\sum_{x\in \RR^1_n}\sum_{y\in \RR^2_n} G_d(x,y).
\end{equation}
The expression \reff{sketch-4}, in conjunction with
the decomposition \reff{dec.intro} which we iterate, 
explains heuristically the form \reff{formal-gammaG}.

The rest of the paper is organised as follows. In Section 2, we start by recalling
known estimates on Green's kernel, and deriving useful simple estimates
on random walks. Then we present the proof of Proposition~\ref{pro:deccap}.
The Strong Law of Large Numbers is established in Section~\ref{sec:slln}
after a rough second estimate is obtained for the cross term. 
Section~\ref{sec:def} studies the limiting object in the CLT.
Section~\ref{sec:nonint} presents our non-intersection events --
Proposition~\ref{lem:inter} which generalises 
Lawler's Theorem~\ref{thm:lawler}. Section 6 presents the
asymptotics for the cross term obtained by the method of moments.
Section~\ref{sec:clt} establishes the CLT based on
estimates of Section 6, and the recursive use of the decomposition.
Finally Section 8 gathers computations linked with Section~\ref{sec:nonint}.

\section{Preliminaries}
\subsection{Notation and standard estimates}

We mostly use the symbol $S$ to denote a random walk, and 
use both notation $S_k$ and~$S(k)$ to denote its position at time $k$.
When $0\leq a\leq b$ are real numbers, $\RR[a,b]$ denotes $\RR[[a],[b]]$, 
where~$[x]$ stands for the integer part of $x$. 
We also write $\RR_a$ for $\RR[0,[a]]$, and $S(n/2)$ for $S([n/2])$.

For positive functions $f,g$ we write $f\sim g$ if $f(n)/g(n)\to 1$ as $n\to\infty$. We also write $f(n) \lesssim g(n)$ if there exists a constant $c > 0$ such that $f(n) \leq c g(n)$ for all $n$,
and $f(n) \gtrsim g(n)$ if $g(n) \lesssim f(n)$.  Finally, we use the notation $f(n) \asymp g(n)$ if both $f(n) \lesssim g(n)$ and $f(n) \gtrsim g(n)$.

For $\alpha>0$, and $n\ge 2$, we note $n_\alpha := n \cdot (\log n)^{-\alpha}$.

The Euclidean norm of $x\in \Z^4$ is denoted $\|x\|$, and the Euclidean ball of center $x$ and radius $r$ is denoted $\mathcal B(x,r)$. We denote by $\bP_x$ the law of a simple random walk starting from $x$, and simply write $\bP$ when $x=0$. 
Likewise $\bP_{x,x'}$ denotes the law of two independent random walks starting from~$x$ and $x'$, and similarly when there are more walks. 
Recall that $H_A$ denotes the hitting time of a set $A$, and we abbreviate this in $H_x$ when $A$ is reduced to a single point $x\in \Z^d$. 

We write 
$$p_k(x,y) = \bP_x(S_k=y).$$
The function $p_k$ is symmetric in $x$ and $y$, and one has $p_k(x,y)=p_k(0,y-x):=p_k(y-x)$. Define 
\begin{equation}
\label{fk}
f_k(x) = \frac{8}{\pi^2k^2}\exp\left(-2\frac{\norm{x}^2}{k} \right).
\end{equation}
A well-known estimate, see Proposition 2.1.2 (b) in \cite{LawLim}, 
shows that for some positive constants $c$ and $C$
\begin{equation}
\label{upper.heatkernel}
\forall k\geq 1,\qquad
\bP\left(\max_{\ell \le k} \, \|S_\ell\|\ge r\right)
\ \le \ C\cdot e^{-c\cdot r^2/k}.
\end{equation}
Furthermore, for any fixed $\alpha<2/3$, one has for all $k\ge 1$ and $x$ with $p_k(x)>0$ and $\|x\|\le k^\alpha$, (see Proposition 1.2.5 in \cite{Law91}): 
\begin{equation}
\label{localCLT}
p_k(x) = f_k(x)(1+ \mathcal O(k^{3\alpha-2})).
\end{equation}
One deduces in particular the following useful estimate: 
\begin{equation}
\label{dispertion}
\bP (\|S_k\|^2\le k/R) \ =  \ \mathcal O(R^{-2}) . 
\end{equation}
The discrete Green's function $G_d$ is defined by 
$$G_d(x,y) \ =\ \sum_{k\ge 0} \bP_x(S_k=y)\quad\text{and if }
x\not= y,\quad G_d(x,y) =\ \bP_x(H_y<\infty)\cdot G_d(0,0).$$
We also write $G_d(x)=G_d(0,x)$, and recall that $G_d$ is symmetric, and satisfies $G_d(x,y)=G_d(y-x)$.

The continuous Green's function $G(x,y)$ is also symmetric and satisfies $G(x,y)=G(0,y-x):=G(y-x)$. 
It is defined for $z\in \R^4$ non zero, by 
 \begin{equation}
\label{Green.cont}
G(z) \ =\ \frac{1}{2\pi^2}\cdot \frac{1}{\|z\|^2}.
\end{equation}
These two functions are linked by the relation (see Theorem 4.3.1 in \cite{LawLim}):
for $x\in \Z^4$ 
\begin{equation}
\label{Green.bound}
G_d(x) = 4G(x) + \mathcal O\left(\frac{1}{1+\|x\|^4}\right).
\end{equation}
We will also use the following 
(see Proposition 6.5.1 and 6.5.2 in \cite{LawLim}): 
there exists a constant $C>0$, such that for all $x$ and $r>0$, 
\begin{equation}
\label{hit.ball}
\bP_x(H_{\mathcal B(0,r)}<\infty) \ \le\ C\cdot \frac{r^2}{1+\|x\|^2}.
\end{equation} 
Finally we prove two useful estimates on the heat kernel $p_k(x)$:

\begin{claim}\label{cl:pik}
	Let $x\in \Z^d$ and $k\in \N$ be such that $\sqrt{k}\leq \|x\|\leq k^{3/5}$. Then there exists a positive constant $C$ (independent of $x$) so that for all $i\leq k$ we have
	\[
	p_i(x) \ \leq\  C\,  f_k(x).
	\]
\end{claim}

\begin{proof}[\bf Proof]
Suppose first that $i\le k^{1-\varepsilon}$, for some $\varepsilon>0$ to be fixed later. Then one can use \eqref{upper.heatkernel} which gives 
\begin{align*}
p_i(x) \, \le\, \bP(\|S_i\|\ge \|x\|)\, \lesssim \, \exp\left(-c\frac{\|x\|^2}{i}\right)\, \lesssim \, \exp\left(-\frac c2 k^{\varepsilon}-2\frac{\|x\|^2}{k}\right)\, \lesssim\,  f_k(x),
\end{align*}
using for the third inequality that for $k$ large enough,  $ \|x\|^2/i \ge \max(k^{\varepsilon},(2/c)\cdot \|x\|^2/k)$.

Suppose next that $k^{1-\varepsilon} \le i\leq k/2$. 
Now choose $\varepsilon$ such that 
$3/(5(1-\varepsilon))<2/3$. Then one can use the local CLT \eqref{localCLT}, which gives
\begin{align*}
	p_i(x) \, \lesssim\,  \frac{1}{i^2}\exp\left( -\frac{2\norm{x}^2}{i} \right)\,  =\,  \frac{\norm{x}^4}{i^2} \exp\left(-\frac{\norm{x}^2}{i} \right) \cdot \frac{1}{\norm{x}^4}\exp\left(-\frac{\norm{x}^2}{i} \right).
\end{align*}
Using that the function $y^2e^{-cy}$ is upper bounded by a constant, the assumption that $i\leq k/2$ and $\norm{x}\geq \sqrt{k}$, we obtain
\begin{align*}
	p_i(x)\lesssim \frac{1}{k^2}\exp\left( - \frac{2\norm{x}^2}{k} \right)
\end{align*}
and this completes the proof in the case $i\leq k/2$.

Suppose finally that $k/2\leq i\leq k$. Then, we conclude the proof
\[
p_i(x) \lesssim \frac{1}{i^2}\exp\left( - \frac{2\norm{x}^2}{i} \right) \lesssim \frac{1}{k^2} \exp\left(-\frac{2\norm{x}^2}{k} \right).
\]
\end{proof}

Recall next the notation $n_\alpha = n/(\log n)^\alpha$.

\begin{claim}\label{cl:usefulpikx}
	Let $i,k,n \in\N$ and $x,z\in \Z^4$ satisfy $k\geq n_\alpha$, $\norm{x}\leq \sqrt{n}(\log n)^2$, $i\leq k/(\log n)^{\beta}$ and $\norm{z}\leq \sqrt{i}(\log n)^{\gamma}$ with $\alpha, \beta$ and $\gamma$ satisfying $\beta>\alpha+4$ and $4+2\gamma+\alpha-\beta<0$. Then we have as~$n\to \infty$
	\[
	f_{k-i}(x-z) = f_k(x) \cdot (1+o(1)).
	\]
\end{claim}

\begin{proof}[\bf Proof]
First, $k$ and $i$ depend on $n$, and as $n\to\infty$
	\[
\frac{1}{k^2} = \frac{1}{(k-i)^2} \cdot (1+o(1)). 
\]
We next turn to the exponential terms in the expression for $f_k$. We have
\begin{align*}
	\frac{\|x\|^2}{k} - \frac{\|x-z\|^2}{(k-i)} = \frac{(k-i)\|x\|^2 - k(\|x\|^2 - 2\langle x,z\rangle + \|z\|^2)}{k(k-i)} = \frac{2k\langle x,z\rangle - k\|z\|^2 - i\|x\|^2}{k(k-i)}.
\end{align*}
It suffices to prove that this last expression tends to $0$ as $n\to\infty$.
By the assumption
\begin{align*}
	\frac{i \|x\|^2}{k(k-i)} \lesssim (\log n)^{\alpha+4-\beta} \quad  \text{ and } \quad  \frac{k\norm{z}^2}{k(k-i)}\lesssim (\log n)^{2\gamma-\beta},
\end{align*}
and since $\alpha +4<\beta$ and $2\gamma<\beta$ they both tend to $0$. Finally, by Cauchy-Schwarz we get
\begin{align*}
	\frac{|\langle x,z\rangle|}{k-i} \leq \frac{\|x\| \|z\|}{k-i} \lesssim (\log n)^{2+\gamma-\frac{\alpha}{2}-\frac{\beta}{2}} \to 0 \,\text{ as } n\to\infty,
\end{align*}
again by using the assumption on $\alpha, \beta$ and $\gamma$ and this completes the proof. 
\end{proof}

\subsection{A decomposition formula for the capacity}\label{sec:new}
Recall the last passage decomposition formula (see for instance Proposition 4.6.4 in \cite{LawLim}): for any finite set $A\subseteq \Z^d$ and $x\notin A$,
\begin{equation}
\label{last.passage}
\bP_x(H_A<\infty) \ =\  \sum_{y\in A} \, \Gd(x,y) \cdot \bP_y(H_A^+= \infty).
\end{equation}
We also recall two well-known formulas for the capacity of a finite set $A\subseteq \Z^d$. First 
\begin{equation}\label{decomp-1}
\cc{A}=\lim_{\norm{x}\to\infty} \frac{\bP_x(H_A<\infty)}{\Gd(x,0)},
\end{equation}
and for any $y\in A$,
\begin{equation}\label{decomp-2}
\frac{\bP_y(H^+_A=\infty)}{\cp(A)}\ =\ \lim_{\norm{x}\to\infty}\ 
\bP_x\big(S(H_A)=y \mid H_A<\infty\big).
\end{equation}
The first formula is obtained through the last passage decomposition formula~\eqref{last.passage} and the definition of the capacity \eqref{def-cap}, and the second is Theorem 2.1.3 of Lawler's book \cite{Law91}.

\begin{proof}[\bf Proof of Proposition \ref{pro:deccap}]

Consider two finite subsets $A$ and $B$ of $\Z^d$. One has 
\begin{align}
\label{decomp-3}
\bP_x(H_{A\cup B}<\infty)= \bP_x(H_A<\infty)+\bP_x(H_B<\infty)-
&\bP_x(H_A<\infty,\, H_B<\infty) \\
\nonumber = \bP_x(H_A<\infty)+\bP_x(H_B<\infty)- 
&\left( \bP_x(H_A<H_B<\infty)+ \bP_x(H_B<H_A<\infty)\right)\\ \nonumber &- \bP_x(H_A=H_B<\infty). 
\end{align}
Consider now the term $\bP_x(H_A<H_B<\infty)$. Conditioning on the possible hitting point in~$A$ and using the Markov property yield:
\begin{equation*}
\begin{split}
\bP_x(H_A<H_B<\infty)=&\sum_{y\in A\bs B}
\bP_x\big(S(H_{A\cup B})=y,\,  H_{A\cup B}<\infty\big)
\bP_y\big(H_B<\infty\big)\\
=&\bP_x\big(H_{A\cup B}<\infty\big) \sum_{y\in A\bs B}
\bP_x\big(S(H_{A\cup B})=y\mid H_{A\cup B}<\infty\big)
\bP_y\big(H_B<\infty\big).
\end{split}
\end{equation*}
Then, use \reff{decomp-1} and \reff{decomp-2} to obtain
\begin{equation*}
\lim_{x\to\infty} \frac{\bP_x(H_A<H_B<\infty)}{\Gd(0,x)}=
\sum_{y\in A\bs B}\bP_y\big(H^+_{A\cup B}=\infty\big) \bP_y\big(H_B<\infty\big).
\end{equation*}
Finally by using the last passage formula \eqref{last.passage}, we get the desired limit
\begin{equation*}
\lim_{x\to\infty} \frac{\bP_x(H_A<H_B<\infty)}{\Gd(0,x)}=
\sum_{y\in A\bs B} \sum_{z\in B} \, \bP_y\big(H^+_{A\cup B}=\infty\big)
 \Gd(y,z) \bP_z\big(H^+_B=\infty\big).
\end{equation*}
By symmetry one also has 
\begin{equation*}
\lim_{x\to\infty} \frac{\bP_x(H_B<H_A<\infty)}{\Gd(0,x)}=
\sum_{y\in B\bs A} \sum_{z\in A}\,  \bP_y\big(H^+_{A\cup B}=\infty\big)
 \Gd(y,z) \bP_z\big(H^+_B=\infty\big).
\end{equation*}
We also obtain the existence of the limit $\varepsilon(A,B)$ of 
$\bP_x(H_A=H_B<\infty)/\Gd(0,x)$, as $x\to\infty$, since 
in \reff{decomp-3} all other
ratios converge. 
To conclude just note that 
\[
\bP_x(H_A=H_B<\infty)\le \bP_x(H_{A\cap B}<\infty),
\] 
which gives
$\epsilon(A,B)\le \cc{A\cap B}$. \end{proof}

We will apply successively the decomposition of Proposition~\ref{pro:deccap}. To this end, we define for $i\ge 1$ and $j\le 2^i$, 
$$
\RR^{(i,j)}_n \ : = \ \RR[(j-1)2^{-i}n, j2^{-i}n].
$$

\begin{proposition}\label{pro:cprange}
	Fix $p\in \N$. 
	 Then we have 
	\[
	\cc{\RR_n} = \sum_{j=1}^{2^p} \cc{\RR_n^{(p,j)}} - \sum_{i=1}^{p}\sum_{j=1}^{2^{i-1}} \chi_n(i,j) - \epsilon_n,
	\]	
	where $\E{\epsilon_n^2}=\mathcal{O}((\log n)^2)$ and 
	$$\chi_n(i,j) := \chi(\RR_n^{(i,2j-1)},\RR_n^{(i,2j)}) +\chi(\RR_n^{(i,2j)},\RR_n^{(i,2j-1)}).$$ 
\end{proposition}

\begin{proof}[\bf Proof]
	The proof follows directly by applying repeatedly Proposition~\ref{pro:deccap} to $\RR_n$. Moreover, from Proposition~\ref{pro:deccap} we have that in every subdivision the term $\epsilon$ is upper bounded by the size of the intersection of two independent ranges. A straightforward calculation shows that this has expectation $\log n$ (see for instance~\cite[Section~3.4]{Law91}). The bound on the second moment follows from~\cite[Lemma~3.1]{LGRosen}.
	Since we are only taking a finite sum, the result follows.
\end{proof}

\section{Strong law of large numbers}\label{sec:slln}
In this section we prove Theorem \ref{thm:lawoflargenumbers}. 
The main part of the proof consists in obtaining good bounds on 
the first and second moments of the cross term $\chi_n(1,1)$ appearing in the decomposition formula of the capacity. 
More precisely we show that
\begin{lemma}\label{lem:firstsecondchi11}
The first and second moments of $\chi_n(1,1)$ satisfy
\begin{align}
\label{chi1A}
\E{\chi_n(1,1)}\ &\lesssim\  n\cdot\frac{\log \log n}{(\log n)^2} \quad \text{ and } 
\\ \label{chi2A}
\E{\chi_n(1,1)^2}\ &\lesssim\  n^2\cdot\frac{(\log \log n)^2}{(\log n)^4}.
\end{align}
\end{lemma}

Then, in Section~\ref{sec:slln}, we deduce the strong law of large numbers by using
our decomposition of the capacity, Proposition \ref{pro:cprange}. 
Section~\ref{sec:prelimslln} is devoted to some preliminary facts needed for the proof of Lemma~\ref{lem:firstsecondchi11}. 

\subsection{Preliminaries}\label{sec:prelimslln}
We first recall a standard fact, which directly 
follows from \eqref{Green.bound}: the mean time a walk spends
in a ball of radius $R$ is of order $R^2$. More precisely 
\begin{equation}\label{fact-1}
\sum_{k\in \N} \bP(\|S(k)\|\le R)
\ = \   \sum_{x\in \B(0,R) }\Gd(0,x)\ =\  \mathcal O(R^2).
\end{equation}
Next we present a lemma which is needed in the second moment
estimate, and which deals with intersecting the trace of a path $\RR_n$
by two independent random walks starting far apart. 
The proof follows basically from estimates of Lawler \cite{Law91}. 
Recall that $n_\alpha = n/(\log n)^\alpha$. 
\begin{lemma}\label{lem-2-inter}
Let $\alpha>0$. Consider three independent random walks $S,S^1,S^2$ starting at the
origin, and let $x,y\in \Z^d$ with $\|x\|^2,\, \|y\|^2\ge n_\alpha$. There is a constant $C>0$, such that
for $n$ large enough,
\begin{equation}\label{ineq-2-inter}
\pr{(x+\RR^1[0,\infty)) \cap \RR_n
\not=\emptyset, \ (y+\RR^2[0,\infty)) \cap \RR_n 
\not=\emptyset }\ \le\ C\cdot \left(\frac{\log\log n}{\log n}\right)^2.
\end{equation}
\end{lemma}

\begin{proof}[\bf Proof] The proof consists in showing that even if the two events considered in \eqref{ineq-2-inter} are not independent, 
one can still dissociate them. 
Consider the two random times
\[
\sigma_x=\inf\{k:\ S(k)\in \, \, x+\RR^1[0,\infty)\},\quad
\text{and}\quad
\sigma_y=\inf\{k:\ S(k)\in \,\, y+\RR^2[0,\infty)\}.
\]
Note that $\sigma_x$ (resp. $\sigma_y$)
is independent of $S^2$ (resp. of $S^1$).
We can express the non-intersection event in terms of $\sigma_x$
and $\sigma_y$:
\[
\pr{(x+\RR^1[0,\infty))\cap \RR_n
\not=\emptyset, \, ((y+\RR^2[0,\infty))\cap \RR_n
\not=\emptyset}\ \le\ 
\pr{\sigma_x\le \sigma_y\le n}+
\pr{\sigma_y\le \sigma_x\le n}.
\]
By symmetry, it is enough to deal with the first probability on the right-hand side. 
Conditioning on $S^1$ and $\sigma_x$, we get
\[
\begin{split}
\pr{\sigma_x\le \sigma_y\le n}\ \le& \ \E{\1(\sigma_x\le n)\
\prstart{\RR[0,n-\sigma_x]\cap (y+\RR^2[0,\infty))
\not= \emptyset}{S(\sigma_x)}}\\
\ \le &\ \E{\1(\sigma_x\le n)\ \prstart{\RR[0,n]\cap (y+\RR^2[0,\infty)) \not= \emptyset}{S(\sigma_x)}}.
\end{split}
\]
Note that since $\|y\|^2\ge n_{\alpha}$, 
we have using \reff{hit.ball}, 
\begin{equation*}
\pr{\sigma_x\leq n, \|S(\sigma_x)-y\|^2\, \le\,  n_{\alpha +2}}\ \leq \pr{H_{\B(y,\sqrt{n_{\alpha+2}})}<\infty} \lesssim \ \frac{1}{(\log n)^2}.
\end{equation*}
Now, on the event 
$\{\sigma_x\leq n, \|S(\sigma_x)-y\|^2\ge n_{\alpha +2}\}$, \cite[Theorem 4.3.3]{Law91} shows that 
\begin{equation*}\label{simple-10}
\prstart{(y+\RR^2[0,\infty))\cap
\RR[0,n] \not=\emptyset}{S(\sigma_x)}\ \lesssim \ \frac{\log\log n}{\log n}. 
\end{equation*}
Thus by another application of \cite[Theorem 4.3.3]{Law91} we get 
\begin{equation*}
\pr{\sigma_x\le \sigma_y\le n}\  \lesssim \ \bP(\sigma_x\le n) \cdot \frac{\log\log n}{\log n} 
\ \lesssim \ \left(\frac{\log\log n}{\log n}\right)^2,
\end{equation*}
and this completes the proof.
\end{proof}

\subsection{First and second Moment estimates of the cross terms}\label{sec:firstslln}

\begin{proof}[\bf Proof of Lemma~\ref{lem:firstsecondchi11}]

Note first that by reversibility of the walk, $\chi_n(1,1)$ is equal in law
to $\chi(\RR_{n/2},\widetilde \RR_{n/2}) + \chi(\til{\RR}_{n/2}, \RR_{n/2})$, with $\RR$ and $\widetilde \RR$ 
the ranges of two independent walks $S$ and $\widetilde S$. By symmetry (and Cauchy-Schwarz) it is enough to bound the first and second moments of $\chi(\RR_{n/2},\widetilde \RR_{n/2})$. 
However, to avoid annoying factors $1/2$ everywhere, we will bound the term $\chi(\RR_{n},\widetilde \RR_{n})$ instead, which is of course entirely equivalent. 
Recall that for any finite sets $A$ and~$B$, we have by definition of $\chi(A,B)$ and using also the last exit formula \eqref{last.passage}: 
$$\chi(A,B)= \sum_{y\in A} \bP_y(H_{A\cup B}^+= \infty)\cdot \bP_y(H_B<\infty).$$
Even though the first moment bound
 \eqref{chi1A} follows from \eqref{chi2A} by using Jensen's inequality, 
it is interesting to include a direct proof of \eqref{chi1A}.
Indeed, it serves as a warmup for the proof of the second moment estimate.

So let us prove \eqref{chi1A}. For this we need to consider two additional independent random walks $S^1$ and~$\widetilde S^1$ starting from the origin and also independent of $S$ and $\widetilde S$. 
Denote their ranges by $\RR^1$ and~$\widetilde \RR^1$ respectively. 
We start with a handy bound: 
\begin{equation*}
\begin{split}
\chi(\RR_n,\widetilde \RR_n)=&\ \sum_{k=1}^n \1(S_k\notin \RR_{k-1}) 
\bP((S_k+\RR^1[1,\infty))\cap (\RR_n\cup \widetilde \RR_n)=
\emptyset,\, (S_k+\widetilde \RR^1[0,\infty))
\cap\widetilde\RR_n\not=\emptyset \mid  S,\, \widetilde S)\\
\le &\  \sum_{k=1}^n \bP\big((S_k+\RR^1[1,\infty))
\cap\RR_n=\emptyset \mid  S\big)\cdot 
\bP\big((S_k+ \widetilde \RR^1[0,\infty))\cap\widetilde \RR_n\not=\emptyset \mid  S,\, \widetilde S\big).
\end{split}
\end{equation*}
Taking expectation on both sides and choosing any $\alpha>2$, 
we get using \eqref{fact-1} at the first line and Lawler's results \cite[Theorem 3.5.1 \text{and }4.3.3]{Law91} at the second and fourth lines:  
\begin{equation*}
\begin{split}
\mathbb E[\chi(\RR_n,\widetilde \RR_n)]\le & \sum_{k=1}^{n}\sum_{\|x\|^2\ge n_\alpha} \bP\big((x+\RR^1[1,\infty))
\cap\RR_n=\emptyset, \, S_k= x\big)
\bP\big((x+ \widetilde \RR^1[0,\infty))\cap\widetilde \RR_n\not=\emptyset\big) + \mathcal O(n_\alpha)\\
\lesssim & \  \sum_{k=1}^{n}  \ \bP\big((S_k+\RR^1[1,\infty))
\cap\RR_n=\emptyset \big)\cdot \frac{\log \log n}{\log n} + \mathcal O(n_\alpha)\\
\lesssim & \sum_{k=n_\alpha}^{n-n_\alpha}  \bP\big((S_k+\RR^1[1,\infty))
\cap\RR_n=\emptyset \big)\cdot \frac{\log \log n}{\log n} + \mathcal O(n_\alpha)\\
\lesssim &\ n\cdot \frac{\log \log n}{(\log n)^2}.
\end{split}
\end{equation*}
This proves \eqref{chi1A}.

When taking the square, we need more notation. So 
let $\{S,\widetilde S,S^1,S^2, \widetilde S^1,\widetilde S^2\}$
be independent random walks all starting at the origin.
Fix $\alpha>4$, and introduce the following event 
\[
\AA_n(k_1,k_2)=\{ \|S(k_1)\|^2\ge n_\alpha\quad \text{and}\quad \|S(k_2)\|^2\ge n_\alpha\}.
\]
To simplify notation, write also $\RR^i$ for $\RR^i[0,\infty)$ and $\RR^i_+$ for $\RR^i[1,\infty)$. 
Now, using \eqref{fact-1} at the second line and Lemma \ref{lem-2-inter} at the fourth line, we arrive at 
\begin{align*}
\E{\chi(\RR_n,\widetilde \RR_n)^2}\le & 
\ \sum_{k_1=1}^n \sum_{k_2=1}^n\, 
\bP\big((S_{k_i}+\RR^i_+)\cap \RR_n=\emptyset,
\, (S_{k_i}+\widetilde \RR^i)\cap \widetilde \RR_n\not=\emptyset,\, \forall i=1,2\big)\\
\le \sum_{k_1, k_2}& \bP\big(\AA_n(k_1,k_2), 
 (S_{k_i}+\RR^i_+)\cap \RR_n=\emptyset,
 (S_{k_i}+\widetilde \RR^i)\cap \widetilde \RR_n\not=\emptyset,  \forall i=1,2\big)+ \mathcal O(n\cdot n_\alpha)\\
\le &\ \sum_{k_1,k_2} \sum_{x_1,x_2}\ 
\bP\big(\AA_n(k_1,k_2), \, S_{k_i}=x_i, \, (x_i+\RR^i_+)\cap \RR_n=\emptyset, \, \forall i=1,2\big)\\
&\qquad \times   \bP\big((x_i+\widetilde \RR^i)\cap \widetilde \RR_n 
\not=\emptyset, \, \forall i=1,2 \big)+\mathcal O(n\cdot n_\alpha)\\
\lesssim &\  \sum_{k_1,k_2}
\bP\big(\AA_n(k_1,k_2),\,  (S_{k_i}+\RR^i_+)\cap \RR_n=\emptyset,\, \forall i=1,2\big)\cdot \left(\frac{\log \log n}{\log n}\right)^2+\mathcal O(n\cdot n_\alpha)\\
\lesssim   &\sum_{n_\alpha \le k_1\le k_2-n_\alpha\le n- 2n_\alpha} 
 \bP\big((S_{k_i}+\RR^i_+)\cap \RR_n=\emptyset,\, \forall i=1,2\big)\cdot \left(\frac{\log \log n}{\log n}\right)^2+\mathcal O(n\cdot n_\alpha).
\end{align*}
We deal now with the non-intersection terms for which we removed the space constraints but we added time constraints. 
From the walk $S$,
we form two walks stemming from position $S_m$ with
$m=[(k_1+k_2)/2]$. One walk goes backward, and is denoted $S^3$, and
another one goes forward and is denoted $S^4$. Translating the origin to
$S_m$ we obtain using \cite[Theorem 3.5.1]{Law91},
\begin{equation*}
\begin{split}
\bP\big((S_{k_i}+\RR^i_+)\cap \RR_n=\emptyset, \, \forall i=1,2\big)
\le &\bP\big( (S^3_{m/2}+\RR^1_+)\cap \RR^3[0,m]=\emptyset\big)\\ &\times \bP\big((S^4_{m/2}+\RR^2_+)\cap \RR^4[0,n-m]=\emptyset\big)
= \ \mathcal O\left(\frac{1}{(\log n)^2}\right),
\end{split}
\end{equation*}
which proves \eqref{chi2A}. 
\end{proof}
 
\subsection{Strong law of large numbers}\label{sec:slln}
We are now ready for the proof of Theorem \ref{thm:lawoflargenumbers}. 
The first step is to 
use the dyadic decomposition of the capacity to produce
self-similar {\it independent} terms at a smaller scale. 
If this scale is well chosen, 
then the result of the previous section shows that the 
cross terms $\chi_n(i,j)$ are negligible, by a simple application of Chebyshev's inequality. 
On the other hand, using Borel-Cantelli and Chebyshev's inequality again, one can show that the self-similar part 
converges almost surely, at least along some subsequence growing sub-exponentially fast. 
Finally, using the monotonicity of the capacity we can deduce the convergence along the full sequence.

\begin{proof}[\bf Proof of Theorem~\ref{thm:lawoflargenumbers}]
Choose $L$ as a function of $n$, such that $(\log n)^4\le 2^L\le 2 (\log n)^4$. In particular one has $L\asymp \log \log n$. 
Proposition \ref{pro:cprange} shows that 
\begin{equation}\label{lln-1}
\cc{\RR_n}=\sum_{j=1}^{2^L} \cc{\RR^{(L,j)}_n} 
-\xi_n-\epsilon_n,
\end{equation}
where $\mathbb E[\epsilon_n]=\mathcal O(2^L\log n) = \mathcal O((\log n)^5)$, and 
$$\xi_n=\sum_{i=1}^L\sum_{j=1}^{2^{i-1}} \chi_n(i,j).$$
Now \eqref{chi2A} shows that 
for all $i\le L$ and $j\le 2^{i-1}$
\begin{equation*}
\mathbb E[\chi_n(i,j)^2]\ \lesssim \ n^2\cdot \frac{(\log \log n)^2}
{2^{2i}(\log n)^4}.
\end{equation*}
Using Cauchy-Schwarz and the independence of the $(\chi_n(i,j))_j$ 
for any fixed $i$, one obtains
\begin{equation*}
\vr{\xi_n} \le  L\cdot \sum_{i=1}^{L}
\vr{\sum_{j=1}^{2^{i-1}}\chi_n(i,j)} \ =\ L \cdot \sum_{i=1}^{L} 2^{i-1} \vr{\chi_n(i,1)} \ = \  \mathcal O\left(n^2\cdot \frac{(\log \log n)^3}{(\log n)^4}\right).
\end{equation*}
This together with Chebyshev's inequality give 
\begin{equation}\label{lln-4}
\pr{| \xi_n-\E{\xi_n}|> \epsilon \frac{n}{\log n}}\ \lesssim\  \frac{(\log \log n)^3}{(\log n)^2},
\end{equation}
where $\epsilon>0$.
On the other hand, using the trivial bound $\cp(\RR_n)\le |\RR_n|\le n$,  
and Chebyshev's inequality again for a sum of independent terms,
we get
\begin{equation}\label{lln-5}
\bP\left(\sum_{j=1}^{2^L}  \left|\cc{\RR^{(L,j)}_n}-\E{\cc{\RR_n^{(L,j)}}}\right|>
\epsilon \frac{n}{\log n}\right)\ \lesssim \ \frac{1}{(\log n)^2}.  
\end{equation}
Now consider the subsequence $a_n=\exp(n^{3/4})$, and observe that 
it satisfies  
\[
\lim_{n\to\infty} \frac{a_{n+1}}{a_n}=1,\quad\text{and}\quad
\sum_{n} \frac{(\log\log a_n)^3}{(\log a_n)^2}<\infty.
\]
Therefore \eqref{chi1A}, \eqref{lln-4} and Borel-Cantelli show that almost surely 
$$\lim_{n\to \infty}\ \frac{\log a_n}{a_n}\cdot  \xi_{a_n} = 0.$$
Similarly using the bound $\E{\epsilon_n}=\mathcal O((\log n)^5)$, we deduce using Markov's inequality that almost surely 
$$\lim_{n\to \infty}\ \frac{\log a_n}{a_n}\cdot  \epsilon_{a_n} = 0.$$
Finally, using \eqref{LLN1moment}, \eqref{lln-1}, \eqref{lln-5} and Borel-Cantelli again, we deduce that almost surely 
$$\lim_{n\to \infty}\ \frac{\log a_n}{a_n}\cdot  \cp(\RR_{a_n}) = \frac{\pi^2}{8}.$$
To conclude, first observe that $\cp(\RR_n)$ is nondecreasing, since for any $A\subset B$, one has 
$\cp(A)\le \cp(B)$. Thus if for $n\ge 1$, we define $k_n$ as the unique integer, such that 
$a_{k_n} \le n < a_{k_n+1}$, we have 
\begin{equation*}
\cc{\RR[0,a_{k_n}]}\ \le\ \cc{\RR[0,n]}\ \le\ \cc{\RR[0,a_{k_n+1}]}.
\end{equation*}
Since $a_{k_n+1}/a_{k_n}$ goes to $1$, as $n\to \infty$, the sequence 
$(\log n/n) \cdot \cc{\RR[0,n]}$ converges.  
\end{proof}

\begin{remark}
{\em Note that for the last part of 
our proof we followed the strategy as Dvoretzky and Erd\"os
in their pioneering work on the range \cite{DE}. They also first proved an almost sure limiting result along a subsequence growing subexponentially fast (using also Chebyshev's inequality and good bounds on the variance), and then 
deduced the strong law of large numbers using a monotonicity argument. 
The idea that a decomposition like \eqref{incl.excl} 
could be useful in obtaining sharp variance bounds and a 
central limit theorem came much later, in Le Gall's papers 
\cite{LG86, LG88}.  }
\end{remark}

\section{Existence and definition of the limiting term}\label{sec:def}
The goal of this section is to give a precise definition of the limiting term appearing in Theorem~\ref{thm:clt}.  
We also prove a Carleman's condition for the sum approximating it.  

\subsection{Carleman's condition}
We recall that Carleman's condition holds for a nonnegative random variable $X$, if its sequence of moments $m_p:=\mathbb E[X^p]$ satisfies: 
$$\sum_{p\ge 1}\ (m_p)^{- \frac 1{2p}} \ = \ \infty.$$
\begin{proposition}\label{pro:carleman}
Let 
$$X = \int_0^1\int_0^1 \frac{1}{\|\beta_s-\widetilde \beta_t\|^2}\, ds\, dt,$$
where $\beta$ and $\widetilde \beta$ are two independent standard $4$-dimensional Brownian motions.  
There exist positive constants $c$ and $C$, such that 
$$c^p \, p^p \ \le \ \mathbb E\left[X^p \right]\ \le\ C^p\cdot  p^{2p},$$
for all $p\ge 1$. In particular the upper bound implies that Carleman's condition holds for $X$. 
\end{proposition}
\begin{remark}
\emph{Note that by reversibility of the Brownian path, and scaling invariance, one has 
$$X \ \stackrel{(d)}{=} \ 2\, \int_0^{1/2} \int_{1/2}^1\ \frac{1}{\|\beta_s- \beta_t\|^2}\, ds\, dt.$$
}
\end{remark}

\begin{proof}[\bf Proof]
We first prove the upper bound. By using Jensen's inequality we get  
$$\mathbb E[X^p] \lesssim \int_0^1\, dt \, \mathbb E \left[ \left( \int_0^1 \, \frac{ds}{\|\beta_s - \widetilde \beta_t\|^2}\right)^p\right]. $$
A symmetry argument gives 
\begin{equation}
\label{moment1}
\mathbb E[X^p] \ \lesssim\ p! \, \int_0^1\, dt \,  \int_{0\le s_1\le \ldots \le s_p\le 1} \, \mathbb E \left[ 
\frac {1}{\|\beta_{s_1}- \widetilde \beta_t\|^2\cdots \|\beta_{s_p}- \widetilde \beta_t\|^2 }\right] \, ds_1\ldots ds_p.
\end{equation}
By using standard properties of the Brownian motion, we can write 
\begin{eqnarray}
\label{moment2}
&&\int_{0\le s_1\le \ldots\le s_p\le 1} \mathbb E \left[ 
\frac {1}{\|\beta_{s_1}- \widetilde \beta_t\|^2\cdots \|\beta_{s_p}- \widetilde \beta_t\|^2 }\right] \, ds_1\dots ds_p\\
\nonumber &=&\int_{0\le s_1\le \ldots\le s_p\le 1}   \mathbb E \left[ 
\frac {1}{\|\beta_{s_1}- \widetilde \beta_t\|^2\cdots \|\beta_{s_{p-1}}- \widetilde \beta_t\|^2 } \, \mathbb E \left(  \frac{1}{\| \beta_{s_p}- \widetilde \beta_t\|^2}\ \Big|\  (\beta_u)_{u\le s_{p-1}}, \widetilde \beta_t \right)   \right]\, ds_1\ldots  ds_p  \\
\nonumber &=&\int_{0\le s_1\le \ldots\le s_{p-1}\le 1} ds_1\ldots ds_{p-1}  \mathbb E \left[ 
\frac {1}{\|\beta_{s_1}- \widetilde \beta_t\|^2\cdots \|\beta_{s_{p-1}}- \widetilde \beta_t\|^2 } \, \int_0^{1- s_{p-1}} \mathbb E \left(  \frac{1}{\| \beta_s- x_p\|^2}\ \Big|\  x_p \right) \, ds  \right],  
\end{eqnarray}
with $x_p = \beta_{s_{p-1}} - \widetilde \beta_t$. 
Now we need the following two lemmas. 
\begin{lemma} 
\label{momentLem1} One has 
\begin{equation*}
\int_0^1 \mathbb E\left[\frac {(a+|\log \|\beta_s - x\||)^k}{\|\beta_s-x\|^2}\right]\, ds \ \lesssim \ \sum_{\ell=0}^{k+1}(4k)^\ell \cdot \left( a+1+ |\log \|x\||\right)^{k+1-\ell}, 
\end{equation*}
uniformly in $x\in \mathbb R^4\setminus \{0\}$, $a\ge 0$ and $k\ge 0$. 
\end{lemma}

\begin{lemma} \label{momentLem2} For all $k\ge 1$ and $a\ge k$, 
$$ \int_0^1 \mathbb E\left[ (a+|\log \| \widetilde \beta_t \||)^k  \right] \, dt \ \lesssim \  k\cdot a^k.$$
\end{lemma}
We prove these two lemmas in the Appendix. Let us now conclude the proof of the proposition. 
First one can use Lemma \ref{momentLem1} and by induction, we can bound the last integral in \eqref{moment2} by 
$$C^p \cdot \sum_{\ell = 0}^p \alpha_\ell \cdot (4p)^{\ell}(p + |\log \|\widetilde \beta_t\||)^{p-\ell},$$
where  $C$ is a constant and 
$$\alpha_\ell \ :=   \#\left\{1\le i_1\le \dots \le i_\ell \le p\right\} \ = \  \frac{p^\ell}{\ell !}  .$$
Then using Lemma \ref{momentLem2}, \eqref{moment1}, and \eqref{moment2} we obtain the upper bound 
$$\mathbb E[X^p] \ \lesssim \ C^p \cdot p^{2p} \cdot  \sum_{\ell = 0}^p \frac{(4p)^\ell}{\ell !} \ \lesssim \ (Ce^4)^p \cdot p^{2p},$$
which proves the upper bound.

Now we prove the lower bound. Define 
$$\Gamma(p):= \inf_{x\in \mathcal B(0,2/\sqrt p)} \inf_{z,z'\in \mathcal B(0,1/\sqrt p)} \int_0^{1/p}\int_0^{1/p} \mathbb E_{z,z'}\left[\frac{\1(\beta_{1/p} \in \mathcal B(0,1/\sqrt p),\, \widetilde \beta_{1/p} \in \mathcal B(0,1/\sqrt p)}{\|x+\beta_s - \widetilde \beta_t\|^2} \right]\, ds\, dt.$$
Note that by scaling $\Gamma(p) = \Gamma(1)/p$, and that 
$$\Gamma(1) \  \geq \  \frac{1}{16}   \, \bP\left(\sup_{s\le 1} \| \beta_s \| \le 1\right)^2:=c_0>0.$$
Then define 
$$\mathcal C : =\{0\le s_1\le \dots \le s_p\le 1\}\quad \text{and}\quad 
\mathcal D = \left[0,\frac 1p\right]\times \left[\frac 1 p,\frac 2 p\right]\times \dots \times \left[\frac{p-1}p,1\right],$$
and 
$$A:= \left\{\beta_{i/p},\, \widetilde \beta_{i/p} \in \mathcal B(0,1/\sqrt p)  \quad  \text{for all }i=1,\dots,p\right\}.$$
Then for any permutation $\sigma$ of the set $\{1,\dots, n\}$, one has by the Markov property and standard properties of Brownian motion,
\begin{eqnarray*}
&&\int_{\mathcal C} \int_{\mathcal C} \, \prod ds_i 
\prod dt_i\  \mathbb E \left[\prod_{i=1}^p \frac 1{\|\beta_{s_i} - 
\widetilde \beta_{t_{\sigma(i)}}\|^2}\right] 
 \ge \int_{\mathcal D} \int_{\mathcal D}\, \prod ds_i \prod dt_i\  \mathbb E \left[\prod_{i=1}^p \frac {\1(A)}{\|\beta_{s_i} - \widetilde \beta_{t_{\sigma(i)}}\|^2}\right]\\
&=& \int_{\mathcal D} \int_{\mathcal D} \, \prod ds_i \prod dt_i\  \mathbb E \left[\prod_{i=1}^p \frac {\1(A)}
{\|x_i + (\beta_{s_i}- \beta_{i/p}) - (\widetilde \beta_{t_{\sigma(i)}} - \widetilde \beta_{\sigma(i)/p})\|^2}\right]\\
&\ge & \Gamma(p)^p\ \ge\ (c_0/p)^p , 
\end{eqnarray*}
with $x_i = \beta_{i/p} - \widetilde \beta_{\sigma(i)/p}$, for $i=1,\dots,p$. 
Now one has 
$$\mathbb E[X^p] \ \ge \ (p!)^2  \ \inf_\sigma\  \int_{\mathcal C} \int_{\mathcal C}\,  \prod ds_i \, \prod dt_i\   \mathbb E \left[\prod_{i=1}^p \frac 1{\|\beta_{s_i} - \widetilde \beta_{t_{\sigma(i)}}\|^2}\right],$$
and this gives the lower bound using the previous bound and Stirling's formula. 
\end{proof}

\subsection{The limiting term}\label{sec:limit}
We have now all the ingredients to define properly the term $\gamma_G([0,1]^2)$, appearing in Theorem \ref{thm:clt}. 
First define the following subsquares of $[0,1]^2$:  
$$A_{i,j} = [(2j-2)2^{-i}, (2j-1)2^{-i}]\times [(2j-1)2^{-i}, (2j) 2^{-i}],$$ 
for $i\ge 1$ and $j\leq 2^{i-1}$. Define also  $\C_1=\{s,t\in[0,1]:\, s\leq t\}$, the closure of the union of all these squares.

A straightforward computation shows that 
\[
\E{\int_0^1\int_0^1 G(\beta_s,\beta_t)\,ds\, dt}=\infty.
\]
 
However, if we consider disjoint intervals, then this expectation is finite as we already proved in Proposition~\ref{pro:carleman}, i.e.
\[
\E{\int_0^{1/2}\int_{1/2}^1 G(\beta_s,\beta_t)\,ds\, dt} <\infty,
\]
and by scaling the same fact holds when integrating over any of the squares $A_{i,j}$. 
This observation is at the heart of the following proposition: 

\begin{proposition}\label{pro:conv}
	Let $\beta$ be a standard $4$-dimensional Brownian motion. Then the following limit exists in $L^2$:   
	\[
	\gamma_G(\C_1)\ :=\ \sum_{i=1}^{\infty}\sum_{j=1}^{2^{i-1}}\, \left( \int_{A_{i,j}} G(\beta_s,\beta_t)\, ds\, dt - \E{\int_{A_{i,j}}G(\beta_s,\beta_t)\, ds\, dt}\right).
	\]
Moreover, there exists $\lambda\in  \R$, such that 
$$\mathbb E\left[ e^{\lambda \cdot \gamma_G(\C_1)}\right] \ =\  \infty. $$
\end{proposition}
\begin{remark}
{\em We take by definition $\gamma_G([0,1]^2) : = 2\, \gamma_G(\C_1)$.}
\end{remark}

\begin{proof}[\bf Proof]
By Brownian scaling we have for any $j\le 2^{i-1}$, 
\begin{align}\label{eq:scaling}
	\int_{A_{i,j}} G(\beta_s,\beta_t)\, ds\, dt\  \stackrel{law}{=} \ \frac{1}{2^{i-1}} \int_0^{1/2}\int_{1/2}^1 G(\beta_s,\beta_t)\, ds\, dt,
\end{align}
and moreover, by Proposition \ref{pro:carleman} these random variables have finite second moment. 
Now by the triangle inequality for the $L^2$-norm we have 
\begin{align*}
&\ \norm{\sum_{i=1}^{n}\sum_{j=1}^{2^{i-1}} \left( \int_{A_{i,j}} G(\beta_s,\beta_t)\,ds\, dt - \E{\int_{A_{i,j}} G(\beta_s,\beta_t)\,ds\, dt } \right)
}_{2} 
\\
&\quad \leq \ \sum_{i=1}^{n}\ \norm{\sum_{j=1}^{2^{i-1}} \left( \int_{A_{i,j}} G(\beta_s,\beta_t)\,ds\, dt - \E{\int_{A_{i,j}} G(\beta_s,\beta_t)\,ds\, dt } \right)
}_2\\
&\quad =\ \sum_{i=1}^{n}\ \left(\sum_{j=1}^{2^{i-1}} \E{\left( \int_{A_{i,j}} G(\beta_s,\beta_t)\,ds\, dt - \E{\int_{A_{i,j}} G(\beta_s,\beta_t)\,ds\, dt }\right)^2} \right)^{1/2},
\end{align*}
where for the last inequality we used the independence of the terms $\int_{A_{i,j}}G(\beta_s,\beta_t)\,ds\, dt$, when the rectangles~$A_{i,j}$  are all in the same subdivision. Using~\eqref{eq:scaling} this last sum is equal to
\begin{align*}
\sum_{i=1}^{n} \sqrt{\frac{2^{i-1}}{2^{2(i-1)}} \cdot \vr{\int_0^{1/2}\int_{1/2}^1G(\beta_s,\beta_t)\,ds\, dt}}= \sqrt{\vr{\int_0^{1/2}\int_{1/2}^1 G(\beta_s,\beta_t)\,ds\, dt}}\sum_{i=1}^{n} \frac{1}{2^{(i-1)/2}},
\end{align*}
which converges as $n\to\infty$.

Now we prove that this random variable has some infinite exponential moment, using the same argument as in Le Gall \cite{LG94}. First note that by definition we can express $\gamma_G(\C_1)$ almost surely as 
$$\gamma_G(\C_1) \  = \ (\gamma_1 + \gamma_2 + X-\mathbb E[X])/2,$$ 
with $\gamma_1$ and $\gamma_2$ two independent random variables with the same law as $\gamma_G(\C_1)$, and $X$ a random variable with the same law as the random variable from Proposition \ref{pro:carleman}. Now the lower bound in the latter proposition shows that $X$ has some infinite exponential moment: there exists $\lambda >0$, such that $\E{\exp(\lambda X)}= \infty$. Since 
$X = 2\gamma_G(\C_1) - \gamma_1-\gamma_2+\mathbb E[X]$, it implies by Cauchy-Schwarz that either 
$\E{\exp(4\lambda \gamma_G(\C_1))}$ or $\E{\exp(-2\lambda \gamma_G(\C_1))}$ is infinite.
\end{proof}

\section{Intersection and non-intersection probabilities}\label{sec:nonint}

The goal of this section is to obtain an asymptotic expression for the probability of non-intersection of a two sided walk with  simple random walk, when one walk is conditioned to end up at a specific location. Our proofs will rely heavily on the following estimate of Lawler on the non-intersection probability of a two sided-walk with a simple random walk. 

\begin{theorem}{\rm (\cite[Corollary 4.2.5]{Law91})}\label{thm:lawler}
Let $\RR^1,\RR^2$ and $\RR^3$ be the ranges of
three independent random walks in $\Z^4$ starting at 0. Then, 
\begin{equation}\label{lawler-key2}
\lim_{n\to\infty} \ \log n\cdot  
\pr{( \RR^1[0,n]\cup \RR^2[0,n])\cap\RR^3[1,\infty)=\emptyset,\ 
0\not\in \RR^1[1,n]}=\frac{\pi^2}{8}.
\end{equation}
\end{theorem}
Recall the definition of $f_k(x)$ from \eqref{fk} and also the shorthand notation $n_\alpha = n/(\log n)^\alpha$. 

\begin{proposition}\label{lem:inter}
	Let $\alpha>8$, $n-n_\alpha>k>n_\alpha$ and  $x \in \Z^4$ with $\sqrt{n_{2\alpha}}\leq \norm{x}\leq \sqrt{n}(\log n)^2$. Let $\RR^1, \RR^2$ and $\RR^3$ be the ranges of three independent random walks starting from $0$. Then	\begin{align*}
		\pr{(\RR^1[0,k]\cup \RR^2[0,n])\cap \RR^3[1,\infty) =\emptyset, S^1(k) = x, 0\notin \RR^1[1,k]} = \frac{\pi^2}{8} \cdot \frac{1}{\log n} \cdot 		p_k(x)\cdot  (1+o(1)) \\+ \mathcal{O}\left(\frac{1}{(\log n)^{3/2}}\cdot 
		f_k\left(x/2\right)\right).
			\end{align*}
\end{proposition}

Note that in the expression above we cannot always absorb the second term in the $o(1)$ term, since~$p_k(x)$ is not always comparable to $f_k(x/2)$. However, this second term is going to be negligible when we take the sum over all time indices and all points in space.

The rest of this section is devoted to the proof of the above proposition. 

The following lemma on the probability that two walks intersect when one walk is conditioned to end up at a specific location is a crucial ingredient in the proof of the proposition above and will be used in later parts too.

\begin{lemma}\label{lem:bruno}
	There exists a positive constant $C$ so that the following holds. Let $z\in 
	\Z^4$ and let~$S^1$ and $S^2$ be two independent simple random walks in $\Z^4$ starting from $0$ and~$z$ respectively. For  $a,k\in \N$ and $b\in \N\cup \{\infty\}$ let 
	$A(a,b,k) = \{\RR^1[0,k]\cap \RR^2[a,b]\neq \emptyset\}$. Then for all $x\in \Z^4$ with $\norm{x}\leq k^{2/3}$ we have 
	\begin{align*}
	\prstart{A(a,b,k), S^1(k)=x}{0,z} \leq Cf_k\left(x/2\right) \cdot \max\left(\prstart{A(a,b,k)}{0,z}, \prstart{A(a,b,k)}{x,z} \right).
	\end{align*}
\end{lemma}

We now state two lemmas and a claim whose proofs are deferred after the proof of Proposition~\ref{lem:inter}.

\begin{lemma}\label{lem:leading}
	Let $\alpha >8$. Let $n_\alpha<k<n-n_\alpha$, $i=k/(\log n)^{5\alpha}$ and   $\sqrt{n_{2\alpha}}\leq \norm{x}\leq  \sqrt{n}(\log n)^2$. Let~$\RR^1, \RR^2$ and $\RR^3$ be the ranges of three independent random walks starting from $0$. Then we have
	\begin{align*}
		\pr{(\RR^1[0,i]\cup \RR^2[0,n])\cap \RR^3[1,\infty) =\emptyset, S^1(k) = x, 0\notin \RR^1[1,i]} = \frac{\pi^2}{8} \cdot \frac{1}{\log n} \cdot 		p_k(x)\cdot  (1+o(1)).
			\end{align*}
\end{lemma}

\begin{lemma}\label{lem:notleading}
	Same assumptions as in Lemma~\ref{lem:leading}. We have
	\begin{align*}
			\pr{(\RR^1[0,i] \cup \RR^2[0,n])\cap \RR^3[1,\infty)= \emptyset, \RR^1[i,k] 
			\cap \RR^3[1,\infty)\neq \emptyset, S^1(k)=x} \lesssim f_k(x/2) \cdot \frac{\log\log n}{(\log n)^{3/2}}.	
	\end{align*}
\end{lemma}

\begin{claim}\label{cl:hitxthen0}
	Let $\alpha>0$, $k>n_\alpha$ and $\|x\| \leq k^{3/5}$. Then we have 
	\[
	\pr{S^1(k)=x, n_\alpha < H_0<k} \lesssim f_k(x) \cdot \frac{(\log n)^{2\alpha}}{n}.
	\] 
\end{claim}

We now give the proof of Proposition~\ref{lem:inter} which is an easy consequence of the results above and then we will prove Lemmas~\ref{lem:bruno}, \ref{lem:leading} and~\ref{lem:notleading} and Claim~\ref{cl:hitxthen0}.

\begin{proof}[\bf Proof of Proposition~\ref{lem:inter}]

Let $i=k/(\log n)^{5\alpha}$ as in Lemma~\ref{lem:leading}. Then we can write
\begin{align*}
	&\pr{(\RR^1[0,k]\cup \RR^2[0,n])\cap \RR^3[1,\infty) =\emptyset, S^1(k) = x, 0\notin \RR^1[1,k]} \\&= \pr{(\RR^1[0,i]\cup \RR^2[0,n])\cap \RR^3[1,\infty) =\emptyset, S^1(k) = x, 0\notin \RR^1[1,i]} \\
	& \quad - \mathbb{P}(\RR^1[0,i]\cup\RR^2[0,n])\cap \RR^3[1,\infty) =\emptyset, S^1(k) = x,0\notin \RR^1[1,i], \\&\quad \quad \quad \quad \quad \quad \quad \quad\quad \quad \quad \quad \{\RR^1[i,k]\cap \RR^3[1,\infty)\neq \emptyset\}\cup\{0\in \RR^1[i,k]\}). 
\end{align*}
For the first term we use Lemma~\ref{lem:leading} to get the asymptotic expression of the statement. Regarding the second probability by the union bound it is upper bounded by
\begin{align*}
	\pr{(\RR^1[0,i]\cup\RR^2[0,n])\cap \RR^3[1,\infty) =\emptyset, S^1(k) = x, \RR^1[i,k]\cap \RR^3[1,\infty)\neq \emptyset} \\ + \pr{S^1(k) = x, 0\notin \RR^1[1,i], 0\in \RR^1[i,k]}.
\end{align*}
The first term can be bounded using Lemma~\ref{lem:notleading} and the second one using Claim~\ref{cl:hitxthen0}.
\end{proof}

\begin{proof}[\bf Proof of Lemma~\ref{lem:bruno}]

We assume that $p_k(x)>0$, since otherwise the statement is trivial.
Suppose first that $\norm{x}^2\leq k$. Then in this case $p_k(x) \asymp 1/k^2$, and hence for all $y,z\in \Z^4$ we have $p_{k/2}(y,z)\le C /k^2 \asymp p_k(x)$. We now get
\begin{align*}
\bP_{0,z}(A(a,b,k), S^1(k)=x) &= \sum_{y} \bP_{0,z}(A(a,b,k),S^1(k)=x, \ S^1(k/2)=y)\\
&\le  \sum_y \bP_{0,z}\left(\RR^2[a,b] \cap \RR^1[0,k/2]\neq \emptyset, \ S^1(k/2)=y, \ S^1(k)=x\right) \\
& \quad+ \sum_y \bP_{0,z}\left(\RR^2[a,b] \cap \RR^1[k/2,k]\neq \emptyset, \ S^1(k/2)=y, \ S^1(k)=x\right).
\end{align*} 
Using a time inversion and the Markov property this last sum is equal to 
 \begin{align*}
 &\sum_y \bP_{0,z}\left(A(a,b,k/2), \ S^1(k/2)=y\right) \cdot  p_{k/2}(y,x) 
 + \sum_y \bP_{x,z}\left(A(a,b,k/2), \ S^1(k/2)=y\right) \cdot  p_{k/2}(y,0)\\
&\lesssim \ p_{k}(x)\cdot  \sum_y \bP_{0,z}\left(A(a,b,k/2), \ S^1(k/2)=y\right) 
 +\, p_{k}(x) \cdot  \sum_y \bP_{x,z}\left(A(a,b,k/2), \ S^1(k/2)=y\right)\\
&\lesssim\ p_{k}(x)  \cdot  \bP_{0,z}\left(A(a,b,k/2)\right)\, + \, p_{k}(x)\cdot   \bP_{x,z}\left(A(a,b,k/2) \right)\\
&\lesssim \ p_{k}(x) \cdot \max(\bP_{0,z}\left(A(a,b,k) \right), \bP_{x,z}\left(A(a,b,k)\right)).
\end{align*}
This completes the proof in the case when $\norm{x}^2\leq k$. 

Suppose next that $\norm{x}^2\geq k$. We write $\B_x = \B(x,\norm{x}/2)$ and $\S_x = \partial \B_x$.
x\begin{align}
\label{hit}
\nonumber &\bP_w( S^1(i) = x, \ \tau_x\le i)  = \sum_{j\le i}\sum_{y\in \S_x} \bP_w(\tau_x = j,\ S^1(j)=y, \ S^1(i)=x)
 \\&=  \sum_{j\le i}\sum_{y\in \S_x}  \bP_w(\tau_x = j,\ S^1(j)=y) \cdot  p_{i-jx}(y,x)
\lesssim \  f_k(x/2)\cdot  \bP_w(\tau_x\le i)\ \lesssim\ f_k(x/2),
\end{align}
where the first inequality follows from Claim~\ref{cl:pik}. Now one can write 
\begin{align}
\label{twoparts}
\begin{split}
\bP_{0,z}(A(a,b,k), S^1(k)=x) \, &\le\,  \bP_{0,z}\left(\RR^2[a,b]\cap \RR^1[0,\sigma_x]\neq \emptyset, \ S^1(k)=x\right) \\ &\quad +\bP_{0,z}\left(\RR^2[a,b] \cap \RR^1[\sigma_x,k]\neq \emptyset, \ S^1(k)=x\right). 
\end{split}
\end{align}
In order to bound the first term, let us define  
$$\mathcal I := \inf\{i\ge 0 \ :\  S^1(i) \in \RR^2[a,b]\}.$$
Note that for any $i$, the event $\{\mathcal I = i\}$ is $\sigma(\RR^1[0,i])\vee \sigma(\RR^2[a,b])$-measurable. 
Therefore by the Markov property we obtain  
 \begin{align*}
&\bP_{0,z}\left(\RR^2[a,b]\cap \RR^1[0,\sigma_x]\neq \emptyset, \ S^1(k)=x\right) =  \sum_{w}\sum_{i\le k} \bP_{0,z}\left(\mathcal I = i, \sigma_x\ge i, \ S^1(i)=w,\ S^1(k)=x\right)\\
&= \sum_{w, \, i} \bP_{0,z}\left(\mathcal I = i, \ S^1(i)=w\right)\cdot \bP_w(S^1(k-i)=x,\, \tau_x\le k-i)
\\&\lesssim \ f_k(x/2) \cdot \sum_{w, \, i} \bP_{0,z}\left(\mathcal I = i, \ S^1(i)=w \right) \\
& \lesssim \, f_k(x/2) \cdot \bP_{0,z}\left(A(a,b,k) \right),
\end{align*}
where we used~\eqref{hit} for the first inequality.
Now concerning the second term in \eqref{twoparts}, one can look at the path backwards in time, and observe that seen from $x$, $\sigma_x$ is now the first hitting time of $\S_x$, namely $\tau_x$. Therefore, using the strong Markov property,  
\begin{align*}
& \bP_{0,z}\left(\RR^2[a,b] \cap \RR^1[\sigma_x,k]\neq \emptyset, \ S^1(k)=x\right)\\
&=  \sum_{y\in \mathcal S_x,\, i\le k} \bP_{0,z}\left(\RR^2[a,b] \cap \RR^1[i,k]\neq \emptyset, \ S^1(i)=y,\ \sigma_x = i, \ S^1(k)=x\right)\\
&= \sum_{y,i} \bP_{x,z}\left(A(a,b,k-i), \ S^1(k-i)=y, \ \tau_x = k-i\right) \cdot  p_i(y)\\
&\lesssim \, f_k(x/2) \sum_{y,i} \bP_{x,z}\left(A(a,b,k), \ S^1(k-i)=y, \ \tau_x = k-i\right) \\
&\lesssim \, f_k(x/2) \cdot \bP_{x,z}\left(A(a,b,k)\right),
\end{align*}
where for the first inequality we used Claim~\ref{cl:pik} again.
This now completes the proof.
\end{proof}

\begin{proof}[\bf Proof of Lemma~\ref{lem:leading}]

This proof is very similar to~\cite[Proposition~4.3.2]{Law91}, but we include it here for the sake of completeness.
Again we assume that $p_k(x)>0$, otherwise the statement is trivial.

We define $D=\{ \|S^1(i)\|\leq \sqrt{i}(\log n)^{\alpha+6}\}$. Then $\pr{D^c}\leq \exp(-(\log n)^{\alpha+6})$. We set 
\[
A = \left\{  (\RR^1[0,i]\cup \RR^2[0,n])\cap \RR^3[1,\infty) =\emptyset, 0 \notin \RR^1[1,i]    \right\}.
\]
Then by Lawler's estimate, Theorem~\ref{thm:lawler}, we have that 
\[
\pr{A} = \frac{\pi^2}{8} \cdot \frac{1}{\log n} \cdot (1+o(1)).
\]
We now obtain
\begin{align*}
\pr{A, S^1(k)=x, D} \leq 	\pr{A, S^1(k) = x} \leq \pr{A, S^1(k) = x, D} + \pr{D^c}.
\end{align*}
By the Markov property we now get
\begin{align*}
	\pr{A\cap D, S^1(k)=x} = \prcond{S^1(k)=x}{A\cap D}{} \pr{A\cap D} = p_k(x)(1+o(1)) \pr{A\cap D},
\end{align*}
where the last equality follows from \eqref{localCLT} and Claim~\ref{cl:usefulpikx}, since after conditioning on $D$, the time changes to~$k-i$ and the walk starts from some $z$ with $\norm{z} \leq \sqrt{i}(\log n)^{\alpha+6}$. Note that for $\alpha>8$ the assumptions of Claim~\ref{cl:usefulpikx} are satisfied.  We also have 
\[
\pr{A\cap D} = \pr{A} - \pr{A\cap D^c}   = \pr{A} (1+o(1)),
\]
since $\pr{A\cap D^c}\leq \pr{D^c} \leq \exp(-(\log n)^{\alpha+6})$ and $\pr{A}\asymp 1/\log n$.
So far we have showed that 
\begin{align*}
	\pr{A\cap D, S^1(k)=x} = p_k(x) (1+o(1)) \pr{A}.
\end{align*}	
By the assumption on the values of $x$ and $k$, we get that 
\[
\pr{A} p_k(x) \gtrsim \exp (-c (\log n)^{\alpha +4}).
\]
Therefore, 
\[
\pr{A\cap D, S^1(k)=x} + \pr{D^c} = p_k(x) \pr{A}(1+o(1))
\]
and this completes the proof.
\end{proof}

\begin{proof}[\bf Proof of Lemma~\ref{lem:notleading}]
Define $\ell = i/(\log n)^{2\alpha+20}$. Then we can upper bound the probability of the statement as follows
\begin{align}
	\nonumber&\pr{(\RR^1[0,i] \cup \RR^2[0,n])\cap \RR^3[1,\infty)= \emptyset, \RR^1[i,k] 
			\cap \RR^3[1,\infty)\neq \emptyset, S^1(k)=x} \\
			\nonumber&\leq \ 
			\pr{\RR^2[0,n]\cap \RR^3[1,\ell] = \emptyset, \RR^1[i,k]\cap \RR^3[1,\infty)\neq \emptyset, S^1(k)=x}\\
			\nonumber&\leq \ \pr{\RR^2[0,n]\cap \RR^3[1,\ell] = \emptyset, \RR^1[i,k]\cap \RR^3[1,\ell]\neq \emptyset, S^1(k)=x}
			 \label{eq:last} \\&\quad +\pr{\RR^2[0,n]\cap \RR^3[1,\ell] = \emptyset, \RR^1[i,k]\cap \RR^3[\ell,\infty)\neq \emptyset, S^1(k)=x}
\end{align}
From Lemma~\ref{lem:bruno} we have 
\begin{align*}
	&\pr{\RR^1[i,k]\cap \RR^3[1,\ell]\neq \emptyset, S^1(k)=x } \\
	&\lesssim \ f_k(x/2)\cdot \max(\prstart{\RR^1[i,k]\cap \RR^3[1,\ell]\neq \emptyset}{x,0}, \pr{\RR^1[i,k]\cap \RR^3[1,\ell]\neq \emptyset}).
\end{align*}
We now define 
\[
D_3=\left\{ \max_{j\leq \ell} \|S^3(j)\|\leq \frac{\sqrt{i}}{(\log n)^4} \right\}.
\]
By the choice of $\ell$ we have $\pr{D_3^c}\leq \exp(-(\log n)^{\alpha+6})$. 
We also let 
\[
D_1=\left\{ \|S^1(i)\|\geq \frac{\sqrt{i}}{(\log n)^2} \right\}. 
\]
Then $\pr{D_1^c}\leq \frac{1}{(\log n)^8}$, and hence we obtain that
\begin{align*}
	\prstart{\RR^1[i,k]\cap \RR^3[1,\ell]\neq \emptyset}{} &\leq \pr{D_3^c} + \pr{D_1^c} + \prstart{\RR^1[i,k]\cap \RR^3[1,\ell]\neq \emptyset, D_1,D_3}{}\\
&	\leq \frac{2}{(\log n)^8} + \max_{\norm{w}\geq \sqrt{i}/(\log n)^2}\prstart{H_{\B\left(0, \frac{\sqrt{i}}{(\log n)^4} \right)}<\infty}{w}\\
&\lesssim \frac{1}{(\log n)^4},
\end{align*}
where for the last inequality we used \eqref{hit.ball}. When $S^1$ starts from $x$, we then get
\begin{align*}
	\prstart{\RR^1[i,k]\cap \RR^3[1,\ell]\neq \emptyset}{x,0} &\leq \pr{D_3^c} + \prstart{\RR^1[i,k]\cap \RR^3[1,\ell]\neq \emptyset, D_3}{x,0}\\&\leq \exp(-(\log n)^{\alpha+6}) + \prstart{H_{\B\left(0, \frac{\sqrt{i}}{(\log n)^4} \right)}<\infty}{x}\\
	&\leq \exp(-(\log n)^{\alpha+6}) + \frac{1}{(\log n)^{8+3\alpha}},
\end{align*}
where for the last inequality we used again~\eqref{hit.ball} and the assumption on $i$ and $\norm{x}$.

So we overall get that 
\[
\pr{\RR^2[0,n]\cap \RR^3[1,\ell] = \emptyset, \RR^1[i,k]\cap \RR^3[1,\ell]\neq \emptyset, S^1(k)=x} \lesssim f_k(x/2)\cdot \frac{1}{(\log n)^4}.
\]
Regarding the probability appearing in~\eqref{eq:last} we obtain
\begin{align}\label{eq:firsttermlast}
\nonumber&	\pr{\RR^2[0,n]\cap \RR^3[1,\ell] = \emptyset, \RR^1[i,k]\cap \RR^3[\ell,\infty)\neq \emptyset, S^1(k)=x} \\
\nonumber&\leq \pr{\RR^2[0,n]\cap \RR^3[1,\ell] = \emptyset, \RR^1[i,k]\cap \RR^3[\ell,\infty)\neq \emptyset, S^1(k)=x, D_3} + \pr{D_3^c} \\
&\lesssim \pr{\RR^2[0,n]\cap \RR^3[1,\ell] = \emptyset, \RR^1[i,k]\cap \RR^3[\ell,\infty)\neq \emptyset, S^1(k)=x, D_3} +f_k(x/2)\cdot \frac{1}{(\log n)^4},
\end{align}
where for the last inequality we used that for $x$ and $k$ as in the statement of the lemma we have
\[
p_k(x) \gtrsim \exp\left(- 2(\log n)^{\alpha+4} \right) \quad \text{ and } \quad \pr{D_3^c} \leq \exp(-(\log n)^{\alpha+6}).
\]
The first term of~\eqref{eq:firsttermlast} is upper bounded by
\begin{align}\label{eq:zilogn}
 \sum_{\|z\|\leq \frac{\sqrt{i}}{(\log n)^4}}\prstart{\RR^1[i,k]\cap \RR^3[1,\infty) \neq \emptyset, S^1(k)=x}{0,z} \pr{\RR^2[0,n]\cap \RR^3[1,\ell]=\emptyset,S^3(\ell)=z}.
\end{align}
Using Lemma~\ref{lem:bruno} again and~\cite[Theorem~4.3.3]{Law91}
we get that for all $z$ in the range as above 
\[
\prstart{\RR^1[i,k]\cap \RR^3[1,\infty) \neq \emptyset, S^1(k)=x}{z,0} \lesssim f_k(x/2) \cdot \frac{\log\log n}{\log n}.
\]
Therefore, the sum of~\eqref{eq:zilogn} becomes upper bounded by
\begin{align*}
	\pr{\RR^2[0,n]\cap \RR^3[1,\ell]=\emptyset} \cdot f_k(x/2) \cdot \frac{\log\log n}{\log n} \asymp f_k(x/2)\cdot \frac{\log \log n}{(\log n)^{3/2}},
\end{align*}
where for the equivalence we used~\cite[Theorem~4.4.1]{Law91}. Substituting this into~\eqref{eq:firsttermlast} finishes the proof.
\end{proof}

\begin{proof}[\bf Proof of Claim~\ref{cl:hitxthen0}]
In order to upper bound the probability of this event we consider two cases, either $\norm{x}\leq \sqrt{k}$ or $\norm{x}>\sqrt{k}$. If $\norm{x}\leq \sqrt{k}$, then using reversibility and the Markov property we obtain
\begin{align*}
	&\pr{S^1(k)=x, n_\alpha<H_0<k} = \pr{S^1(k)=x, n_\alpha<H_0\leq\frac{k}{2}} + \pr{S^1(k)=x, \frac{k}{2}	<H_0<k} \\
	&=\sum_z \pr{S^1(k/2)=z, S^1(k)=x, n_\alpha<H_0\leq \frac{k}{2}} + \sum_z \pr{S^1(k/2)=z, S^1(k)=x, \frac{k}{2}<H_0<k} \\
	&\lesssim\,  p_k(x) \pr{n_\alpha<H_0\leq \frac{k}{2}} + p_k(x)\prstart{H_0<\infty}{x}\lesssim\ p_k(x) \cdot 
	\frac{(\log n)^{2\alpha}}{n}.
\end{align*}
We turn to the case $\norm{x}>\sqrt{k}$. We now have using the Markov property 
\begin{align*}
	\pr{S^1(k)=x, n_\alpha<H_0<k} &= \prcond{S^1(k)=x}{n_\alpha<H_0<k}{}\pr{n_\alpha<H_0<k} \\&\leq \sup_{i<k} p_i(x)\cdot \pr{n_\alpha<H_0<k} 
	 \lesssim\,   f_k(x) \cdot \frac{(\log n)^{2\alpha}}{n},
\end{align*}
where for the last inequality we used Claim~\ref{cl:pik}.
\end{proof}


\section{Joint convergence in law of the cross terms}
The main purpose of this section is to prove the joint convergence in law of the cross terms. 
First recall the definition of the squares 
$$A_{i,j} = [(2j-2)2^{-i}, (2j-1)2^{-i}]\times [(2j-1)2^{-i}, (2j) 2^{-i}],$$ 
for $i\ge 1$ and $j\leq 2^{i-1}$. 
Recall also the definition of the cross terms from Proposition~\ref{pro:cprange}: $$\chi_n(i,j)= \chi(\RR_n^{(i,2j-1)}, \RR_n^{(i,2j)}) + \chi(\RR_n^{(i,2j)}, \RR_n^{(i,2j-1)}),$$
with 
$$ \RR_n^{(i,j)} = \RR[(j-1)2^{-i}n,j2^{-i}n].$$
\begin{proposition}
\label{prop.moments.chi}
Let $p\ge 1$ be a fixed integer. Then as $n\to \infty$, 
\begin{align}
\label{conv.law.cross}
	\left(\frac{2(\log n)^2}{\pi^4\cdot n}\cdot \chi_n(i,j) \right)_{1\le i\le p,\ 1\le j\le 2^{i-1}} \quad 
\stackrel{(d)}{\Longrightarrow} \quad \left(\int_{A_{i,j}}  G(\beta_s, \beta_t)\, ds\, dt \right)_{1\le i\le p,\ 1\le j\le 2^{i-1}}.
\end{align}
Moreover, all moments converge.
\end{proposition}
Our strategy for proving this proposition is first to localize in a certain sense all the $\chi_n(i,j)$. 
More precisely we prove that for any given $i$ and $j\le 2^{i-1}$, the term $\chi_n(i,j)$ can be written as a sum of two elements, one being 
a localised version of this cross term (the so-called $\chi_{n,\alpha}(i,j)$, see below), and the other one having a negligible expectation. 
So to prove the joint convergence 
in law of the cross terms, we are led to prove only the joint convergence in law of the $\chi_{n,\alpha}(i,j)$. 
To prove this in turn, we show the convergence of the joint moments, which is indeed sufficient thanks to the results of Section 
4.1 and Carleman's criterion (see Section 3.3.5 in \cite{Durrett}).

Now the $\chi_{n,\alpha}(i,j)$ have the great advantage that their joint moments 
reduce (after some tedious computation) to a product of non-intersection probabilities (whose asymptotics have been 
computed in the previous section) times a product of Green's function. 
Then a separate argument shows that this product of Green's functions converges to its continuous counterpart.   

Before digging into the proof, we gather some basic preliminary estimates in the next subsection. 

\subsection{Preliminaries}
We start with an elementary fact which directly follows from the local CLT \eqref{localCLT} and~\eqref{upper.heatkernel}: 
there is a constant $C>0$, such that for all $k\ge 0$,
\begin{equation}\label{GSK-1}
\E{G_d(S_k)}\ \le\  C \cdot \frac{1}{k+1} \quad \text{and}\quad \E{G_d(S_k)^2}\ \le\  C \cdot \frac{1 }{k^2+1}.
\end{equation}
Now for $\alpha>0$, recall that $n_\alpha = n/(\log n)^\alpha$, and define the event 
\begin{equation}
\label{Balpha}
B_\alpha=\{(x,y):\, \sqrt{n_{2\alpha}}\leq \|x\|,\|y\|\leq \sqrt{n} (\log n)^2, \|x-y\|\geq \sqrt{n_{2\alpha}}\}.
\end{equation}
\begin{lemma}\label{lem:gg}
Let $S$ and $\widetilde S$ be two independent random walks and let $\alpha>0$.
Then 
\begin{equation}\label{GSK-2}
\sum_{k=0}^{n} \sum_{\ell  =0}^{n_\alpha}\ 
\E{ G_d\big(S_k-\widetilde S_\ell\big)}\  \lesssim\   n_\alpha \cdot  \log n , 
\end{equation}
\begin{equation}\label{GSK-2bis}
\sum_{k=0}^{n}\sum_{\ell =n-n_\alpha}^{n} \ 
\E{ G_d\big(S_k-\widetilde S_\ell\big)}\  \lesssim\   n_\alpha, 
\end{equation}
and 
\begin{equation}\label{GSK-3}
\sum_{k,\ell=0}^{n}\sum_{(x,y)\notin B_\alpha } \bP( S_k=x, \, \widetilde S_\ell=y) \cdot 
G_d(x,y)\ \lesssim\  n_{2\alpha}\cdot (\log n)^2. 
\end{equation}
\end{lemma}
\begin{proof}[\bf Proof]
Note that $S_k-\widetilde S_\ell$ is equal in law to $S_{k+\ell}$. Thus 
by using \reff{GSK-1}, we deduce that for any~$k\ge 1$,
\begin{equation*}
\sum_{\ell=0}^{n_\alpha} \ \E{ G_d\big(S_{k+\ell}\big)}\ \lesssim\ \sum_{\ell=0}^{n_\alpha} \frac{1}{k+\ell}
\ \lesssim \ \frac{n_{\alpha}}{k}.
\end{equation*}
Summing over $k$ proves~\eqref{GSK-2}. The proof of~\eqref{GSK-2bis} is entirely similar.

We prove now \eqref{GSK-3}. Using \eqref{dispertion} yields 
$$\bP(\|S_k - \til{S}_\ell \|^2 \le n_{2\alpha})\, =\, \bP(\|S_{k+\ell} \|^2 \le n_{2\alpha})\, \lesssim\   \frac{n_{2\alpha}^2}{1+(k+\ell)^2}.$$
Similarly one has 
$$\bP(\|S_k\|^2 \le n_{2\alpha})\, \lesssim\   \frac{n_{2\alpha}^2}{1+k^2} \quad \text{and} \quad \bP(\|\til{S}_\ell\|^2 \le n_{2\alpha})\, \lesssim\   \frac{n_{2\alpha}^2}{1+\ell^2}.$$
Therefore by using \eqref{GSK-2} and Cauchy-Schwarz, we get 
\begin{equation*}
\mathbb E\big[ G_d(S_k , \til{S}_\ell) \1 ((S_k, \til{S}_\ell)\notin B_\alpha) \big] \ \lesssim\   \frac{n_{2\alpha}}{1+k+\ell} \left( \frac{1}{1+k} + \frac 1{1+\ell}\right).
\end{equation*}
The result follows by summing over $k$ and $\ell$, and using also \eqref{upper.heatkernel}.
\end{proof}

\begin{lemma}\label{lem:epsilonphi}
For all $i\geq 1$ there exists a constant $C>0$, such that for all $j\le 2^{i-1}$, one has 
$$\frac 1n\cdot  \E{\sum_{(k,\ell) \in A_{i,j}^n} \1(\|S_k-S_\ell \|\le \varepsilon \sqrt n) \cdot G_d(S_k,S_\ell) } \ \le \ C\cdot \varepsilon \log \frac 1\varepsilon.$$
\end{lemma}

\begin{proof}[\bf Proof]
By considering two independent random walks we get
\begin{align}\label{ob-4}
\begin{split}
\E{\sum_{(k,\ell) \in A_{i,j}^n} \1(\|S_k-S_\ell \|\le \varepsilon \sqrt n) \cdot G_d(S_k,S_\ell) } \lesssim \frac{1}{n}
\sum_{k=1}^n k \sum_{z:\|z\|\le \sqrt{ \epsilon n}}
p_k(z) \Gd(z)\\ 
\lesssim \frac{1}{n}\sum_{k=1}^n 
\sum_{z:\|z\|\le \sqrt{ \epsilon n}}
k\times \frac{\exp(-\|z\|^2/(2k))}{k^2} \Gd(z).
\end{split}
\end{align}
Summing over $z$ we now obtain
\[
\sum_{z:\|z\|\le \sqrt{ \epsilon n}} \exp(-\|z\|^2/(2k)) \Gd(z)\lesssim
\int_0^{\sqrt{ \epsilon n}} \frac{\exp(-r^2/(2k))}{r^2} r^3\ dr=
\ k(1-\exp(-\epsilon n/(2k))).
\]
Summing over $k$ we get
\begin{align*}
	\frac{1}{n}\sum_{k=1}^n  \big(1-\exp(-\epsilon n/2k)\big) \leq \epsilon  + \frac{1}{n}
\sum_{k=\epsilon n}^n\epsilon\frac{n}{2k}\lesssim \epsilon \log(1/\epsilon)
\end{align*}
and this concludes the proof.
\end{proof}

\begin{lemma}\label{lem:sumoff}
	We have
	\[
	\sum_{x,y\in \Z^d}\sum_{k\leq n}\sum_{\ell \leq n} f_k(x) f_\ell(y/2) \Gd(x,y) \lesssim n.
	\]
\end{lemma}

\begin{proof}[\bf Proof]

The proof follows immediately by substituting the expression for $f$ and using~\eqref{Green.bound}.
\end{proof}

\subsection{Localisation of one cross term}
\label{sec:deloc}

For any $n\ge 1$, $i\ge 1$ and $j\le 2^{i-1}$, define 
$$A_{i,j}^n\ :=\ \{(2j-2)2^{-i}n,\dots,(2j-1)2^{-i}n\}\times \{(2j-1)2^{-i}n,\dots,j2^{-i+1}n \},$$and for any $\alpha>0$, 
\begin{align*}
	A_{i,j}^{n,\alpha} :&= \{(2j-2)2^{-i}n+n_\alpha,\dots,(2j-1)2^{-i}n - n_\alpha\}\times  \{(2j-1)2^{-i}n+n_\alpha,\dots,j2^{-i+1}n - n_\alpha\}
	\end{align*}
at least for $n$ large enough, to make sense of this definition, and with the usual convention to take integer parts when needed.

\begin{lemma}\label{lem:deloc}
Let $i,j$ be positive integers with $j\le 2^{i-1}$. For all $\alpha>3$ we have 
\[
\chi_n(i,j)=
2\cdot \chi_{n,\alpha}(i,j) + \epsilon_{n,\alpha}(i,j),
\]
where $\E{\epsilon_{n,\alpha}(i,j)} = o(n/(\log n)^2)$ and 
\begin{align*}
\chi_{n,\alpha}(i,j) =&\sum_{(k,\ell) \in A_{i,j}^{n,\alpha}}\sum_{(x,y)\in B_\alpha}
\prcond{\RR[k-n_{4\alpha}, k+n_{4\alpha}]\cap (x+\RR^1)
=\emptyset, \, S_k=x\notin \RR[k-n_{4\alpha},k)}{S}{} \\
& \times \, \Gd(x,y)\cdot \prcond{\RR[\ell-n_{4\alpha}, \ell+n_{4\alpha}]\cap (y+\RR^2)
=\emptyset, \, S_\ell=y\notin \RR[\ell-n_{4\alpha},\ell)}{S}{},
\end{align*}
with $B_\alpha$ as in \eqref{Balpha} and $\RR^1$ and $\RR^2$ the ranges of two independent 
random walks starting from $0$ and where for simplicity we used the convention $\RR^1=\RR^1[1,\infty)$ and similarly for $\RR^2$.
\end{lemma}

In order to prove the lemma above, we first approximate $\chi_n(i,j)$ by
an expression without localisation which we call $\widetilde \chi_{n,\alpha}(i,j)$, and which decorrelates the two parts of the range in some sense.

\begin{lemma}\label{lem:firstsum}
With the same notation as in Lemma~\ref{lem:deloc}, we have
\[
\chi_n(i,j)= 2\cdot 
\til{\chi}_{n,\alpha}(i,j)+\til{\epsilon}_{n,\alpha}(i,j),
\]
where $\E{\widetilde \epsilon_{n,\alpha}(i,j)}=o(n/(\log n)^2)$ and 
\begin{align*}
\til{\chi}_{n,\alpha}(i,j)=& \sum_{\substack{(k,\ell)\in A_{i,j}^{n,\alpha} \\ (x,y)\in B_\alpha}} \prcond{\RR[(2j-1)2^{-i}n,j 2^{-i}n] \cap (y+\RR^2)
=\emptyset, \, S_\ell=y\notin \RR[(2j-1)2^{-i}n,\ell)}{S}{}
\\
 \times \, \Gd(x,y)&\cdot \prcond{\RR[(2j-2)2^{-i}n,(2j-1)2^{-i}n]\cap (x+\RR^1)
=\emptyset, S_k=x\notin \RR[(2j-2)2^{-i}n,k)}{S}{}. 
\end{align*}
\end{lemma}

\begin{proof}[\bf Proof of Lemma~\ref{lem:firstsum}]
 
Since $i$ and~$j$ are going to
be kept fixed while $n$ will tend to infinity, we will not lose generality in doing the proof for $i=0$ and $j=1$. 
Also by moving the origin to $S(n)$, and looking at the range $\RR[0,n]$ backwards, 
one is led to consider $\chi(\RR_n,\widetilde \RR_n) + \chi(\til{\RR}_n,\RR_n)$, with 
$\RR_n$ and $\til{\RR}_n$ two independent ranges.  So it suffices to treat the term~$\chi(\RR_n,\til{\RR}_n)$. 

By the independence of $\RR^1$ and $\RR^2$ we get
\begin{align*}
	\chi(\RR_n,\widetilde \RR_n)= \sum_{k,\ell=0}^{n}\sum_{x,y}\Gd(x,y) & \cdot\mathbb{P}((\RR_n\cup \til{\RR}_n)\cap (x+\RR^1) = \emptyset, \, S_k=x\notin \RR[0,k)\mid S,\til{S}) \\ 
	&\times \bP(\til{\RR}_n\cap (y+\RR^2)=\emptyset,\, \til{S}_\ell=y\notin \til{\RR}[0,\ell)\mid \til{S}).
\end{align*}
Lemma~\ref{lem:gg} shows that we can restrict the sum over $n_\alpha\leq k, \ell \leq n-n_\alpha$ and
$(x,y) \in B_\alpha$
at a cost of at most~$n_\alpha \log n$ in expectation. The probability term appearing above is equal to 
\begin{align*}
&\prcond{\RR_n\cap (x+\RR^1) = \emptyset, S_k=x\notin \RR[0,k)}{S}{}\cdot \prcond{\til{\RR}_n\cap (y+\RR^2)=\emptyset, \til{S}_\ell=y\notin \til{\RR}[0,\ell)}{\til{S}}{} \\
&\quad - \mathbb{P}\left(\RR_n\cap (x+\RR^1) = \emptyset, \til{\RR}_n\cap (y+\RR^2)=\emptyset, \til{\RR}_n\cap(x+\RR^1)\neq \emptyset, \right.\\ 
&\quad \quad \quad\quad \quad \quad \quad \quad \quad \hspace{5cm} \left. S_k=x\notin \RR[0,k),\til{S}_\ell=y\notin \til{\RR}[0,\ell)\vert\, S,\til{S}\right).
\end{align*}
The first term is equal to the probability term in the expression of $\widetilde \chi_{n,\alpha}$. 
So we now turn to the second term. On the event $\{S_k=x\}$, by moving the origin to point $x$ and reversing time we can write
\[
\{\RR_n\cap (x+\RR^1)= \emptyset\}\ =\ \{(\RR^3[0,k]\cup \RR^4[0,n-k]) \cap \RR^1= \emptyset\},
\]
where $\RR^3$ and $\RR^4$ are the ranges of two independent walks starting from $0$. Applying the same to~$\til{\RR}$ we get that the event under consideration is contained in the intersection of the following events
\begin{align}\label{eq:starequation}
\begin{split}
	&\{(\RR^3[0,k]\cup \RR^4[0,n-k]) \cap \RR^1=\emptyset, \,S^3(k)=-x\}\\
	&\{(\til{\RR}^3[0,\ell]\cup \til{\RR}^4[0,n-\ell]) \cap \RR^2=\emptyset, \, \til{S}^3(\ell)=-y\}\\
	&\{(\til{\RR}^3[0,\ell]\cup \til{\RR}^4[0,n-\ell]) \cap (x-y+\RR^1)\neq \emptyset, \, \til{S}^3(\ell)=-y\}.
\end{split}	
\end{align}
Setting $i=n_\beta$ with~$\beta =10\alpha+4$ we next define
\begin{align*}
	A_1 &= \{ (\RR^3[0,k]\cup \RR^4[0,n-k]) \cap \RR^1[0,i]=\emptyset, S^3(k)=-x\}\\
	A_2&=  \{\til{\RR}^3[0,\ell]\cap \RR^2=\emptyset, \til{S}^3(\ell)=-y\}\\
	A_3&=\{\til{\RR}^4[0,n-\ell] \cap \RR^2=\emptyset\}\\
	A_4&=\{\til{\RR}^3[0,\ell] \cap (x-y+\RR^1[0,i])\neq \emptyset, \til{S}^3(\ell)=-y\}\\
	A_5&=\{\til{\RR}^3[0,\ell] \cap (x-y+\RR^1[i,\infty))\neq \emptyset, \til{S}^3(\ell)=-y\}\\
	A_6&=\{\til{\RR}^4[0,n-\ell]) \cap (x-y+\RR^1[0,i])\neq \emptyset\}\\
	A_7&=\{\til{\RR}^4[0,n-\ell]) \cap (x-y+\RR^1[i,\infty))\neq \emptyset\}.
\end{align*}
The first event in \eqref{eq:starequation} is contained in $A_1$; the second one is contained in $A_2\cap A_3$, and the third one 
is contained in the union of $A_4$, $A_5$, $A_6$ and $A_7$. Therefore we get that the probability of the intersection of the events appearing in~\eqref{eq:starequation} is upper bounded by 
\begin{align}\label{eq:sumofprob}
	\pr{A_1, A_2,A_6} + \pr{A_1, A_2, A_7}+\pr{A_1, A_3, A_4 } + \pr{A_1, A_3, A_5}.
\end{align}
The first step now is to decorrelate the events where the range $\RR^1$ appears. Lemma~\ref{lem:bruno} gives 
\begin{align}\label{eq:intprob}
\begin{split}
	&\bP(A_4) =\pr{\til{\RR}^3[0,\ell] \cap (x-y+\RR^1[0,i])\neq \emptyset, \til{S}^3(\ell)=-y}
	\lesssim f_{\ell}(y/2) \times\\ &\max\left(\pr{\til{\RR}^3[0,\ell] \cap (x-y+\RR^1[0,i])\neq \emptyset}, \prstart{\til{\RR}^3[0,\ell] \cap (x-y+\RR^1[0,i])\neq \emptyset}{-y,0}\right).
\end{split}
\end{align}
Defining the event $D=\{\|S_1(r)\|\leq \sqrt{i}(\log n)^{\alpha +2}, \, \forall\, r\leq i\}$ we get from \eqref{upper.heatkernel} 
\begin{align}\label{eq:ddecay}
\pr{D^c}\ \lesssim\  \exp\left(-(\log n)^{2\alpha +4} \right).
\end{align}
On the event $D$, in order for  $\til{\RR}^3[0,\ell]$ and $x-y+\RR^1[0,i]$ to intersect,~$\til{S}^3$ must hit a ball centred at $x-y$ of radius $\sqrt{i}(\log n)^2$ or a ball centred at $x$ of the same radius (depending on whether we start from~$0$ or~$-y$).
Since $\|x-y\|\geq \sqrt{n_{2\alpha}}$ and $\|x\|\geq \sqrt{n_{2\alpha}}$,  
using~\eqref{hit.ball} we get 
\begin{align}\label{eq:logbound}
\pr{\til{\RR}^3[0,\ell] \cap (x-y+\RR^1[0,i])\neq \emptyset}\vee \prstart{\til{\RR}^3[0,\ell] \cap (x-y+\RR^1[0,i])\neq \emptyset}{-y,0}\lesssim \frac{1}{(\log n)^4},
\end{align}
where the last inequality follows from the choice of $\beta$ (recall that $i=n_\beta$ with $\beta=10\alpha +4$). 

Using~\eqref{eq:intprob}, \eqref{eq:logbound} and the independence between $\til{\RR}^3$ and $\til{\RR}^4$ we get
\begin{align*}
	\pr{A_1, A_3, A_4} \lesssim \frac{1}{(\log n)^4}\cdot p_k(x) f_\ell(y/2)\cdot\pr{A_3}\leq\frac{1}{(\log n)^4}\cdot p_k(x) f_\ell(y/2).
\end{align*}
Similarly we get the same upper bound for $\pr{A_1, A_2,A_6}$. It remains to bound the probabilities $\pr{A_1,A_2,A_7}$ and $\pr{A_1, A_3, A_5}$. By the independence between the walks again we get 
\[
\pr{A_1, A_2, A_7} = \pr{A_2} \pr{A_1, A_7} \quad \text{ and } \quad \pr{A_1, A_3, A_5} = \pr{A_3} \pr{A_1, A_5}.
\]
For the probability of the event $A_2$ by exactly the same proof as in Lemma~\ref{lem:leading} we have for a suitable $\gamma>0$
\begin{align*}
	\pr{A_2} \ \leq \ \pr{\til{\RR}^3\left[0,\frac{\ell}{(\log n)^\gamma}\right]\cap \RR^2 = \emptyset,\, \til{S}_3(\ell) = -y}\ \lesssim\   p_\ell(y) \cdot \frac{1}{\sqrt{\log n}},
\end{align*}
where in the last inequality we also used~\cite[Theorem~4.4.1]{Law91}.
By~\cite[Theorem~4.4.1]{Law91} again we get
\[
\pr{A_3}\lesssim \frac{1}{\sqrt{\log n}}.
\]
We now upper bound the probability $\pr{A_1, A_5}$. The probability $\pr{A_1, A_7}$ can be bounded using similar ideas. Recall the definition of the event $D$ above. Then from~\eqref{eq:ddecay} and the independence between the two walks we get
\[
\pr{D^c, S^3(k)=-x, \til{S}^3(\ell)=-y}\leq \frac{1}{(\log n)^4} p_k(x) p_\ell(y),
\]
since for the range of $x$ and $k$ that we are looking at we have $p_k(x) \gtrsim \exp(-2(\log n)^{\alpha+4})$.
Conditioning on $S^1(i)$, the events $A_1$ and $A_5$ become independent, and hence we obtain
\begin{align}\label{eq:a1a5}
	\pr{A_1, A_5, D} \leq  \sum_{\|z\|\leq \sqrt{i}(\log n)^{\alpha+2}} \prcond{A_1}{S^1(i)=z}{} \pr{S^1(i)=z} \prcond{A_5}{S^1(i)=z}{}.
\end{align}
From Lemma~\ref{lem:bruno} again and~\cite[Theorem~4.3.3]{Law91} we obtain for all $z$ with $\|z\|\leq \sqrt{i}(\log n)^{\alpha+2}$
\begin{align*}
	\prcond{A_5}{S^1(i)=z}{} \lesssim \frac{\log \log n}{\log n} \cdot f_\ell(y/2).
\end{align*}
Plugging this into~\eqref{eq:a1a5} gives
\begin{align*}
	\pr{A_1,A_5,D} \lesssim f_\ell(y/2) \frac{\log \log n}{\log n} \pr{A_1}\lesssim f_\ell(y/2) p_k(x)\frac{\log \log n}{(\log n)^2} + \mathcal{O}\left( \frac{\log \log n}{(\log n)^{5/2}}\cdot f_k(x/2)\cdot f_\ell(y/2)  \right),
\end{align*}
where for the last inequality we used Proposition~\ref{lem:inter}. 
Therefore we conclude that the sum of probabilities appearing in~\eqref{eq:sumofprob} is upper bounded by 
\begin{align*}
	\frac{1}{(\log n)^{7/3}}\cdot (p_k(x) p_\ell(y)+ p_k(x) f_\ell(y/2) + f_k(x/2)f_\ell(y/2)).
\end{align*}
Taking the sum over all $k,\ell$ and $x,y$ and applying Lemma~\ref{lem:sumoff} completes the proof.
\end{proof}

\begin{proof}[\bf Proof of Lemma~\ref{lem:deloc}]

Using Lemma~\ref{lem:firstsum} it suffices to prove that for all $i,j\in \N$ we have
\[
\til{\chi}_{n,\alpha}(i,j) = \chi_{n,\alpha}(i,j) + \epsilon_{n,\alpha}(i,j),
\]
where $\E{\epsilon_{n,\alpha}(i,j)}=o(n/(\log n)^2)$. As in the proof of the previous lemma we only prove the result for $i=0$ and $j=1$, and by using reversibility of the walk, we are led to consider two independent ranges $\RR_n$ and $\widetilde \RR_n$ between times $0$ and $n$. 

We now define
\begin{align*}
	H= \left\{\RR_n\cap (x+ \RR^1)= \emptyset, S_k=x\notin \RR [0,k) \right\}.
\end{align*}
Then we have that $H=H_1\cap H_2$, where
\begin{align*}
	H_1&=\{ \RR[k-n_{4\alpha},k+n_{4\alpha}]\cap (x+\RR^1)=\emptyset, S_k=x\notin \RR[k-n_{4\alpha},k) \} \\
	H_2&=\{ (\RR[0,k-n_{4\alpha}]\cup \RR[k+n_{4\alpha},n])\cap (x+\RR^1)=\emptyset, x\notin \RR[0,k-n_{4\alpha})\}.
\end{align*}
Since $\pr{H} = \pr{H_1}+\pr{H_1\cap H_2^c}$, using Lemma~\ref{lem:sumoff} it suffices to prove that for all~$x$ and $k$ satisfying $\sqrt{n_{2\alpha}}\leq\norm{x}\leq \sqrt{n}(\log n)^2$ and $n_\alpha \le k\le  n-n_\alpha$, we have 
\begin{align}\label{eq:goalb1}
	\pr{H_1\cap H_2^c} \lesssim p_k(x) \cdot \frac{\log \log n}{(\log n)^{3/2}} + \mathcal{O}\left(f_k(x/2)/(\log n)^{3/2}	\right).
\end{align}
So we now turn to prove~\eqref{eq:goalb1}. We first note that $H_1\cap H_2^c\subseteq F_1\cup F_2 \cup F_3$, where
\begin{align*}
	F_1 &= \{\RR[0,k-n_{4\alpha}]\cap (x+\RR^1)\neq \emptyset, \, \RR[k, k+n_{4\alpha}]\cap (x+\RR^1)=\emptyset, \, S_k=x\}\\
	F_2&=\left\{ \RR[k+n_{4\alpha}, n] \cap (x+\RR^1)\neq \emptyset,\, \RR[k-n_{4\alpha},k]\cap (x+\RR^1)=\emptyset,\,  S_k=x\right\}\\
	F_3&= \left\{S_k=x\in \RR[0,k-n_{4\alpha}]\right\}.
\end{align*}
We start by proving the upper bound of~\eqref{eq:goalb1} for $\pr{F_1}$. The probability $\pr{F_2}$ can be treated in exactly the same way. 

First we decorrelate the two events appearing in $F_1$, by conditioning on $S_k=x$ and also by considering $\RR^1$ in separate time intervals just like we did in the proof of Lemma~\ref{lem:firstsum}. Let $i=n_{10\alpha+4}$. Then subtracting $x$ and reversing time we obtain $F_1\subseteq F_1(1)\cup F_1(2)$, where 
\begin{align*}
	F_1(1)&= \{\RR^3[n_{4\alpha},k]\cap \RR^1[1,i]\neq \emptyset, \, S^3(k)=-x \}\\
	F_1(2)&=\{\RR^3[n_{4\alpha},k]\cap \RR^1[i,\infty)\neq \emptyset, \RR^4[0,n_{4\alpha}]\cap \RR^1[1,i]=\emptyset, S^3(k)=-x \}, 
\end{align*}
and $\RR^3$ and $\RR^4$ are two independent ranges.
From Lemma~\ref{lem:bruno}
we get 
\begin{align*}
\bP(F_1(1)) \lesssim  \,f_k(x/2) \cdot\max(\prstart{\RR^3[n_{4\alpha},k]\cap \RR^1[1,i]\neq \emptyset}{-x,0},\prstart{\RR^3[n_{4\alpha},k]\cap \RR^1[1,i]\neq \emptyset}{}).
\end{align*}
Just like in the proof of Lemma~\ref{lem:firstsum} we define the event $D=\{\norm{S^1(r)}\leq \sqrt{i}(\log n)^2, \ \forall r\leq i\}$. Then on the event $D$ in order for $\RR^3$ and $\RR^1[1,i]$ to intersect, the range $\RR^3$ must hit the ball centered at $0$ of radius $\sqrt{i} (\log n)^2$. By the choice of $x$, this now gives us
\begin{align*}
\max(\prstart{\RR^3[n_{4\alpha},k]\cap \RR^1[1,i]\neq \emptyset}{-x,0},\prstart{\RR^3[n_{4\alpha},k]\cap \RR^1[1,i]\neq \emptyset}{})\lesssim \frac{1}{(\log n)^4},
\end{align*}
and hence
\[
\pr{F_1(1)} \lesssim f_k(x/2)\cdot \frac{1}{(\log n)^4}.
\]
We now turn to bound $\pr{F_1(2)}$. Clearly $\pr{F_1(2), D^c} \leq p_k(x)/(\log n)^4$. Now on the event $D$ conditioning on $S^1(i)$ we get
\begin{align*}
	\pr{F_1(2),D}\leq  \sum_{\norm{z}\leq \sqrt{i}(\log n)^2} \prcond{\RR^3[n_{4\alpha},k]\cap \RR^1[i,\infty)\neq \emptyset,\,  S^3(k)=-x}{S^1(i)=z}{}\\ \times \pr{\RR^4[0,n_{4\alpha}]\cap \RR^1[1,i]=\emptyset, S^1(i)=z}.
\end{align*}
From Lemma~\ref{lem:bruno} again we obtain 
\begin{align*}
	\prcond{\RR^3[n_{4\alpha},k]\cap \RR^1[i,\infty)\neq \emptyset, S^3(k)=-x}{S^1(i)=z}{}\lesssim \frac{\log \log n}{\log n} \cdot f_k(x/2).
\end{align*}
Plugging this above we finally deduce
\begin{align*}
	\pr{F_1(2),D} \leq \frac{\log \log n}{\log n} \cdot f_k(x/2) \cdot \frac{1}{\sqrt{\log n}},
	\end{align*}
where the last estimate follows from~\cite[Theorem~4.4.1]{Law91}. 

To finish the proof we only need to upper bound $\pr{F_3}$. By reversing time again we obtain
\[
\pr{F_3}\leq \pr{S_k=-x, n_{4\alpha}<H_0<k}\lesssim f_k(x/2)\cdot \frac{(\log n)^{2\alpha}}{n},
\]
where we recall that $H_0$ stands for the first hitting time of $0$ and the last inequality follows from Claim~\ref{cl:hitxthen0}.
Therefore, putting all these bounds together proves~\eqref{eq:goalb1} and this now completes the proof.
\end{proof}

\subsection{Moments of $\chi$}\label{sec:moments}
In this subsection we prove the following result  
\begin{lemma}\label{lem:moments}
Let $r\in \N$ and let $i_1,j_1,\dots,i_r,j_r$ be integers such that $j_m\le 2^{i_m-1}$ for all $m\le r$ (and possibly with repetition). Then for all $\alpha>12$ as $n\to\infty$ we have
\begin{equation}\label{mom-1}
\left(\frac{8\log n}{\pi^2}\right)^{2r}\cdot 
\E{\prod_{m=1}^r\chi_{n,\alpha}(i_m,j_m)}\ \sim \ \E{\prod_{m=1}^r \left(\sum_{(k,\ell)\in A_{i_m,j_m}^{n} }  \Gd(S_k,S_\ell)\right) }.
\end{equation}
\end{lemma}

Before proving the lemma above we state a result on the convergence of discrete quantities to their continuous counterparts. We defer the proof after we prove Lemma~\ref{lem:moments}. For all $\beta>0$ we define the sets $D_\beta$ and $E_\beta$ to be the set of time indices and points in space at distance $n_\beta$ and $\sqrt{n_{2\beta}}$ apart respectively. More precisely, we define
\begin{align}\label{eq:defD}
\begin{split}
D_\beta = \bigg\{(k_1,\ell_1,&\ldots, k_r,\ell_r) \in \N^{2r}:\,\,(k_{m},\ell_m)\in A_{i_m,j_m}^{n,\beta},\,\,\forall \,m\leq r \text{ and } \\
&|k_m - k_{m'}|, |k_m - \ell_{m'}|, |\ell_m-\ell_{m'}|\geq n_\beta,  \,\forall\, m, m'\leq r\text{ with } m\neq m'\bigg\}.
\end{split}
\end{align}
and also, with $B_\beta$ as in \eqref{Balpha},
\begin{align}\label{eq:defE}
\begin{split}
E_\beta=\bigg\{(x_1,y_1,&\ldots, x_r,y_r)\in (\Z^{d})^{2r}: \, (x_m,y_m)\in B_\beta, \,\, \forall \, m\leq r\text{ and }\\
&\norm{x_m - x_{m'}}, \norm{x_m-y_{m'}}, \norm{y_m-y_{m'}}\geq \sqrt{n_{2\beta}}, \, \,\forall \, m,m'\leq r \text{ with } m\neq m'\bigg\}.
\end{split}
\end{align}

\begin{lemma}\label{lem:donsker}
Let $r\in \N$ and let $i_1,j_1,\ldots, i_r, j_r$ be integers satisfying $j_m\leq 2^{i_m-1}$ for all $m\leq r$ (and possibly with repetition). Then as $n\to \infty$ we have 
\begin{align*}
	\frac{1}{n^r}\cdot \prod_{m=1}^{r}\left(\sum_{(k,\ell)\in A_{i_m,j_m}^n} \Gd(S_k,S_\ell) \right) \quad \stackrel{(d)}{\Longrightarrow}\quad  \prod_{m=1}^{r}\int_{A_{i_m,j_m}} 16 G(\beta_s,\beta_t)\,ds\,dt.
\end{align*}
Moreover, for all $\beta>2$ we have 
\begin{align*}
	\frac{1}{n^r} \cdot \sum_{ D_\beta} \1((S_{k_m},S_{\ell_m})_{m\le r} \in E_\beta)\prod_{m=1}^{r} \Gd(S_{k_m},S_{\ell_m}) \quad \stackrel{(d)}{\Longrightarrow} \quad \prod_{m=1}^{r}\int_{A_{i_m,j_m}} 16G(\beta_s,\beta_t)\,ds\,dt,
\end{align*}
where in the sum above we take $(k_1,\ell_1,\ldots, k_r,\ell_r)\in D_\beta$.
Finally in both cases we also have convergence in expectation.
\end{lemma}

\begin{proof}[\bf Proof of Lemma~\ref{lem:moments}]
For shorthand notation we write for $k\ge 0$, $x\in \mathbb Z^d$ and $\Lambda\subset \Z^d$, 
\begin{align*}
\AA_\alpha(k,x, \Lambda)	=\{\RR[k-n_{4\alpha},k+n_{4\alpha}] \cap (x+\Lambda)=\emptyset, S_k=x\notin \RR[k-n_{4\alpha},k)\}. 
\end{align*}
Now for multi-indices $i=(i_1,\dots,i_r)$ and $j=(j_1,\dots,j_r)$, write  as above $A_{i,j}^{n,\alpha} = A_{i_1,j_1}^{n,\alpha}\times \ldots\times A_{i_r,j_r}^{n,\alpha}$. Let $({\RR}^m)_m$ and $(\til{\RR}^m)_m$ be the ranges of  independent walks starting from $0$.  Then we have 
\begin{align}\label{eq:prodexp}
\begin{split}
	\E{\prod_{m=1}^{r}\chi_{n,\alpha}(i_m,{j_m})} = \sum_{\substack{A_{i,j}^{n,\alpha},\, B_\alpha} }  &\pr{\bigcap_{m=1}^{r}(\AA_\alpha(k_m,x_m,{\RR}^m)\cap \AA_\alpha(\ell_m,y_m, \til{\RR}^m))}\\ &\times  \prod_{m=1}^{r} \Gd(x_m,y_m),
	\end{split}
\end{align}
where in the sum above we take $(k_1,\ell_1,\dots,k_r,\ell_r)\in A_{i,j}^{n,\alpha}$ and $(x_m,y_m)\in B_\alpha$ for all $m\leq r$.

First of all it is obvious that
\begin{align}\label{eq:lowdalpha}
\begin{split}
	\E{\prod_{m=1}^{r}\chi_{n,\alpha}(i_m,{j_m})}\geq \sum_{ D_\alpha, E_\alpha}  &\pr{\bigcap_{m=1}^{r}(\AA_\alpha(k_m,x_m,{\RR}^m)\cap \AA_\alpha(\ell_m,y_m, \til{\RR}^m))}\\&\times\prod_{m=1}^{r} \Gd(x_m,y_m).
	\end{split}
\end{align}
Next we want to establish that for a suitable $\beta>\alpha$ we have 
\begin{align}\label{eq:powersgoal}
\begin{split}
	\E{\prod_{m=1}^{r}\chi_{n,\alpha}(i_m,{j_m})}\leq \sum_{ D_\beta, E_\beta}  \pr{\bigcap_{m=1}^{r}(\AA_\beta(k_m,x_m,{\RR}^m)\cap \AA_\beta(\ell_m,y_m, \til{\RR}^m))}\\ \times \prod_{m=1}^{r} \Gd(x_m,y_m)
	 + o(n^r/(\log n)^{2r}).
	\end{split}
\end{align}
Indeed, we can decompose~\eqref{eq:prodexp} into two parts, one over the set $D_\beta\times E_\beta$ and one over the complementary set. Since $\beta>\alpha$, we notice that $\AA_\alpha(k,x,\Lambda)\subseteq \AA_\beta(k,x,\Lambda)$ for all $k,x,\Lambda$. So we only need to show that the sum over $D_\beta^c\cup E_\beta^c$ is $o(n^r/(\log n)^{2r})$. 

Forgetting about the intersection events, we can upper bound this sum by
\begin{align*}
	\sum_{\substack{A_{i,j}^{n,\alpha},\, B_\alpha, \,D_\beta^c}} \pr{\bigcap_{m=1}^{r}\{S(k_m)=x_m,S(\ell_m)=y_m\}}&\cdot \prod_{m=1}^{r}\Gd(x_m,y_m) \\
	+ \sum_{\substack{A_{i,j}^{n,\alpha},\, B_\alpha, \, E_\beta^c}} &\pr{\bigcap_{m=1}^{r}\{S(k_m)=x_m,S(\ell_m)=y_m\}}\cdot \prod_{m=1}^{r}\Gd(x_m,y_m).
\end{align*}
We start by bounding the first sum appearing above. Using that $\Gd(x,y)\leq 1/n_{2\alpha}$ for all $(x,y)\in B_\alpha$ gives
\begin{align}\label{eq:replace}
\begin{split}
	\sum_{A_{i,j}^{n,\alpha}, \,B_\alpha, \,D_\beta^c} \pr{\bigcap_{m=1}^{r}\{S(k_m)=x_m,S(\ell_m)=y_m\}}\cdot \prod_{m=1}^{r}\Gd(x_m,y_m)\\
	\lesssim \frac{n^{2r-1}}{(n_{2\alpha})^{r}} \cdot n_\beta = \frac{n^r}{(\log n)^{\beta-2\alpha r}}, 
	\end{split}
\end{align}
where we used that if $\beta>\alpha$, then $ A_{i_m,j_m}^{n,\alpha}\subseteq 
A_{i_m,j_m}^{n,\beta}$.
We now turn to the second sum. Using again the bound on $\Gd(x,y)$ for $(x,y)\in B_\alpha$ as above, we now get
\begin{align}\label{eq:replace2}
\begin{split}
	\sum_{A_{i,j}^{n,\alpha}, \,B_\alpha, E_\beta^c }\pr{\bigcap_{m=1}^{r}\{S(k_m)=x_m,S(\ell_m)=y_m\}}\cdot \prod_{m=1}^{r}\Gd(x_m,y_m)\\
	\leq \frac{n^{2r}}{(n_{2\alpha})^{r}}\cdot 
	\sup_{t\geq n_\alpha}\pr{\norm{S_{t}}\leq \sqrt{n_{2\beta}}},
	\end{split}
\end{align}
where $S$ stands for a simple random walk on $\Z^4$. Using now~\eqref{dispertion} we can upper bound the second sum by
\begin{align*}
	\sum_{A_{i,j}^{n,\alpha}, \,B_\alpha, E_\beta^c }\pr{\bigcap_{m=1}^{r}\{S(k_m)=x_m,S(\ell_m)=y_m\}}\cdot \prod_{m=1}^{r}\Gd(x_m,y_m) &\lesssim \frac{n^r}{(\log n)^{4\beta-2\alpha-2\alpha r}}.
\end{align*}
Thus taking $\beta >2\alpha r+2\alpha + 2r+1$ proves~\eqref{eq:powersgoal}.

We next show that for all $\beta$ sufficiently large we have the following: for any $0=k_0<k_1<\ldots<k_r$ and any $0=x_0, x_1,\ldots, x_r$ satisfying $|k_{i+1}-k_i|\geq n_\beta$, $\norm{x_{i+1}-x_i}\geq \sqrt{n_{2\beta}}$ and $\norm{x_i}\leq \sqrt{n}(\log n)^2$ for all $i<r$ we have
\begin{align}\label{eq:firstst}
	\begin{split}
		\pr{\bigcap_{m=1}^{r}\AA_\beta(k_m,x_m,\RR^m)} = \left(\frac{\pi^2}{8}\cdot\frac{1}{\log n}\right)^r\cdot \prod_{m=1}^{r}p_{k_m-k_{m-1}}(x_{m-1},x_m) \cdot(1+o(1)) \\+ \mathcal O\left( \frac{1}{(\log n)^{r+\frac{1}{2}}} \cdot \prod_{m=1}^{r}f_{k_m-k_{m-1}}((x_{m-1}-x_m)/2)  \right).
	\end{split}
\end{align}

We write $j=k_{r-1}+n_{4\beta}$ and define the event 
\[
D = \left\{ \norm{S_j-S_{k_{r-1}}} \leq \sqrt{n_{4\beta}} \cdot (\log n)^{3\beta/4}\right\}.
\]
We therefore obtain 
\begin{align}\label{eq:bbcom}
\begin{split}
	\pr{\bigcap_{m=1}^{r}\AA_\beta(k_m,x_m,\RR^m)} = \sum_{z\in B(x_{r-1})}\pr{\bigcap_{m=1}^{r}\AA_\beta(k_m,x_m,\RR^m), S_j=z} \\+\pr{\bigcap_{m=1}^{r}\AA_\beta(k_m,x_m,\RR^m),  D^c},
	\end{split}
\end{align}
where $B(x_{r-1}) = \{z:\,\|z-x_{r-1}\|\leq \sqrt{n_{4\beta}} \cdot (\log n)^{3\beta/4}\}$. 
For $z\in B(x_{r-1})$ by the Markov property we have
\begin{align*}
	\pr{\bigcap_{m=1}^{r}\AA_\beta(k_r,x_r,\RR^r), S_j=z} = \prcond{\AA_\beta(k_r,x_r,\RR^r)}{S_j=z}{}\pr{\bigcap_{m=1}^{r-1}\AA_\beta(k_m, x_m, \RR^m), S_j=z}.
\end{align*}
Taking $\beta$ satisfying $\beta>3$ and 
applying Proposition~\ref{lem:inter} to the first term appearing on the right hand-side above we obtain as~$n\to\infty$
\begin{align}\label{eq:xpz}
\begin{split}
	\prcond{\AA_\beta(k_r,x_r,\RR^r)}{S_j=z}{} &= \frac{\pi^2}{8}\cdot \frac{1}{\log n} \cdot p_{k_r-j}(z,x_r) \cdot (1+o(1)) \\ &\quad \quad + \mathcal{O}\left( \frac{1}{(\log n)^{3/2}} \cdot f_{k_r-j}((z-x_r)/2)   \right).
	\end{split}
\end{align}
Then by Claim~\ref{cl:usefulpikx} and~\eqref{localCLT} and taking $\beta$ also satisfying $\beta>8$ we get that 
\[
p_{k_r-j}(z,x_r) = p_{k_r-k_{r-1}}(x_{r-1},x_r)\cdot (1+o(1))
\]
and similarly for $f$.
Substituting this into~\eqref{eq:xpz} gives
\begin{align*}
\prcond{\AA_\beta(k_r,x_r,\RR^r)}{S_j=z}{} = \frac{\pi^2}{8}\cdot \frac{1}{\log n} \cdot p_{k_r-k_{r-1}}(x_{r-1},x_r) \cdot (1+o(1)) \\ + \mathcal{O}\left( \frac{1}{(\log n)^{3/2}} \cdot f_{k_r-k_{r-1}}((x_r-x_{r-1})/2)   \right).
\end{align*}
Hence overall we obtain
\begin{align*}
\pr{\bigcap_{m=1}^{r} \AA_\beta(k_m,x_m,\RR^m), D}=	\sum_{z\in B(x_{p-1})}\pr{\bigcap_{m=1}^{r-1}\AA_\beta(k_m,x_m,\RR^m), \AA_\beta(k_r,x_r,\RR^r), S_j=z} \\= \frac{\pi^2}{8}\cdot \frac{1}{\log n} \cdot p_{k_r-k_{r-1}}(x_{r-1},x_r) \cdot (1+o(1))\cdot \pr{\bigcap_{m=1}^{r-1}\AA_\beta(k_m,x_m,\RR^m), D}\\
	\quad + \mathcal{O}\left( \pr{\bigcap_{m=1}^{r-1}\AA_\beta(k_m,x_m,\RR^m), D}\frac{1}{(\log n)^{3/2}} \cdot f_{k_r-k_{r-1}}((x_r-x_{r-1})/2)   \right).
\end{align*}
We now turn to the last probability appearing above. First we write
\[
\pr{\bigcap_{m=1}^{r-1}\AA_\beta(k_m,x_m,\RR^m), D} = \pr{\bigcap_{m=1}^{r-1}\AA_\beta(k_m,x_m,\RR^m)}- \pr{\bigcap_{m=1}^{r-1}\AA_\beta(k_m,x_m,\RR^m),D^c}.
\]
By the Markov property and \eqref{upper.heatkernel} we have
\begin{align*}
	\pr{\bigcap_{m=1}^{r-1}\AA_\beta(k_m,x_m,\RR^m),D^c}& \leq \prod_{m=1}^{r-1}p_{k_m-k_{m-1}}(x_{m-1},x_m) \cdot \exp\left( - (\log n)^{2\beta}\right)\\
	&=  o(1)\cdot \pr{\bigcap_{m=1}^{r-1}\AA_\beta(k_m,x_m,\RR^m)}.
\end{align*}
So far we have shown that 
\begin{align*}
	\pr{\bigcap_{m=1}^{r}\AA_\beta(k_m,x_m,\RR^m), D} &= \frac{\pi^2}{8\log n}p_{k_r-k_{r-1}}(x_{r-1},x_r) \cdot \pr{\bigcap_{m=1}^{r-1}\AA(k_m,x_m,\RR^m)}(1+o(1))\\
	&+\mathcal{O}\left( \pr{\bigcap_{m=1}^{r-1}\AA_\beta(k_m,x_m,\RR^m)}\frac{1}{(\log n)^{3/2}} \cdot f_{k_r-k_{r-1}}((x_r-x_{r-1})/2)   \right).
\end{align*}
Next we treat the second term on the right hand-side of~\eqref{eq:bbcom}. Taking $\beta$ such that $3\beta/2>\beta+7$ we have 
\begin{align*}
	\pr{\bigcap_{m=1}^{r}\AA(k_m,x_m,\RR^m), D^c}&\leq \prod_{m=1}^{r-1}p_{k_m-k_{m-1}}(x_{m-1},x_m) \cdot \exp\left(-(\log n)^{3\beta/2} \right) \\
	&\leq \prod_{m=1}^{r-1}p_{k_m-k_{m-1}}(x_{m-1},x_m) \cdot \exp(-(\log n)^{\beta+7}) \\
	\leq \prod_{m=1}^{r}&p_{k_m-k_{m-1}}(x_{m-1},x_m) \cdot \exp(-(\log n)^{2})= o(1)\cdot \pr{\bigcap_{m=1}^{r}\AA(k_m,x_m,\RR^m)},
\end{align*} 
where for the second inequality we used  that $p_{k_r-k_{r-1}}(x_{r-1},x_r) \gtrsim \exp(-2(\log n)^{\beta+4})$, since we have taken $k_r-k_{r-1}\geq n_\beta$. This now proves that
\begin{align*}
	\pr{\bigcap_{m=1}^{r}\AA(k_m,x_m,\RR^m)} &=\frac{\pi^2}{8}\cdot \frac{1}{\log n} \cdot p_{k_r-k_{r-1}}(x_{r-1},x_r) \cdot \pr{\bigcap_{m=1}^{r-1}\AA(k_m,x_m,\RR^m)}\cdot (1+o(1))
	\\ &\,\,+ \mathcal{O}\left( \pr{\bigcap_{m=1}^{r-1}\AA_\beta(k_m,x_m,\RR^m)}\frac{1}{(\log n)^{3/2}} \cdot f_{k_r-k_{r-1}}((x_r-x_{r-1})/2)   \right).
\end{align*}
Iterating this proves~\eqref{eq:firstst}. 

Plugging now~\eqref{eq:firstst} into~\eqref{eq:lowdalpha} and~\eqref{eq:powersgoal} and taking the sum over all time indices and points in space and invoking Lemma~\ref{lem:donsker} finishes the proof of the lemma.
\end{proof}

\subsection{Proof of Proposition \ref{prop.moments.chi}}
\label{sec:factor}

In this section we give the proof of Proposition~\ref{prop.moments.chi}. We start by recalling a result from~\cite{Asselah:2016hc} and then we prove Lemma~\ref{lem:donsker}.

\begin{lemma}\label{lem:previouspaper}
	There exists a positive constant $C$ so that for all $r\geq 1$ and all $i_1,j_1,\ldots, i_r, j_r$ (satisfying $j_m\leq 2^{i_m-1}$ for all $m\leq r$ and possibly with repetition), one has 
	\[
	 \E{\prod_{m=1}^{r}\sum_{(k,\ell)\in A_{i_m,j_m}^n} \Gd(S_k,S_\ell) } \leq C n^r. 
	\]
\end{lemma}

\begin{remark}
\rm{
Since we allow repetition of the indices, Lemma~\ref{lem:previouspaper} shows that the random variables
\[
\frac{1}{n^r}\cdot \prod_{m=1}^{r}\sum_{(k,\ell)\in A_{i_m,j_m}^n} \Gd(S_k,S_\ell) 
\]
are bounded in $L^p$ for all $p\in \N$.
}	
\end{remark}

\begin{proof}[\bf Proof of Lemma~\ref{lem:previouspaper}]

The proof of the lemma follows directly by the Cauchy-Schwarz inequality together with~\cite[Lemma~3.2]{Asselah:2016hc}.\end{proof}

\begin{proof}[\bf Proof of Lemma~\ref{lem:donsker}]
	
	We start by proving the first statement of the lemma.
Let $\epsilon>0$ and $\phi_\epsilon$ be a continuous function satisfying 
	\[
	\1(\norm{x}\geq \epsilon)\leq \phi_\epsilon(x) \leq \1(\norm{x}\geq \epsilon/2).
	\]
Recall from~\eqref{Green.bound} that 
 	\[
 	\Gd(x) = 4G(x) + \mathcal{O}\left(\frac{1}{1+\norm{x}^4}\right).
 	\]
 	To simplify notation we now write 
\begin{align*}
X(m) = \sum_{(k,\ell)\in A_{i_m,j_m}^n} \Gd(S_k,S_\ell), \quad &X_\epsilon(m)= \sum_{(k,\ell)\in A_{i_m,j_m}^n} \Gd(S_k,S_\ell)\phi_\epsilon\left(\frac{S_k-S_\ell}{2\sqrt{n}}\right) \quad \text{ and }\\ &\til{X}_\epsilon(m) = \sum_{(k,\ell)\in A_{i_m,j_m}^n} 4\,G(S_k,S_\ell)\phi_\epsilon\left(\frac{S_k-S_\ell}{2\sqrt{n}}\right).
\end{align*}
Using~Lemma~\ref{lem:previouspaper} it is straightforward to check that 
\begin{align}\label{eq:boundxtilx}
\frac{1}{n^r}\cdot \E{\prod_{m=1}^{r}X_\epsilon(m) - \prod_{m=1}^{r}\til{X}_\epsilon(m)} = \mathcal{O}(1/n).
\end{align}
Note that the function $G(x) \phi_\epsilon(x)$ is continuous and bounded,  and that by Donsker's invariance principle $(S_{[nt]}/2\sqrt n,t\ge 0)$ converges in law to a standard Brownian motion. Hence we obtain as $n\to\infty$ 	
\begin{align}\label{eq:donskereps}
		\frac{1}{n^r} \cdot \prod_{m=1}^{r} \til{X}_\epsilon(m) \quad \stackrel{(d)}{\Longrightarrow} \quad  \prod_{m=1}^{r}\int_{A_{i_m,j_m}}16\,G(\beta_s,\beta_t)\phi_\epsilon(\beta_s-\beta_t)\,ds\,dt.
	\end{align}
For all $m\leq r$ we have
\begin{align*}
	X(m)- {X}_\epsilon(m) \leq \sum_{(k,\ell)\in A_{i_m,j_m}^n} \Gd(S_k,S_\ell) \1(\norm{S_k-S_\ell}\leq \epsilon\sqrt{n}),
\end{align*}
and hence using~Lemma~\ref{lem:epsilonphi} we get
\begin{align}\label{eq:errortermxy}
	\E{X(m)-{X}_\epsilon(m)} \lesssim n \cdot \epsilon \log\left(\frac{1}{\epsilon}
	\right).
	\end{align}
For each $\epsilon>0$ we now define
\[
R_n(\epsilon) = \frac{1}{n^r}\cdot \prod_{m=1}^{r}X(m) - \frac{1}{n^r} \cdot \prod_{m=1}^{r}{X}_\epsilon(m).
\]
In view of~\eqref{eq:boundxtilx} and~\eqref{eq:donskereps}, in order to complete the proof of the first statement, it suffices to prove that $\E{R_n(\epsilon)}\to 0$ as $\epsilon\to 0$ uniformly in $n$.

With the definitions above we can upper bound $R_n(\epsilon)$ by
\begin{align*}
	R_n(\epsilon) \leq \frac{1}{n^r} \cdot \sum_{w=1}^{r} (X(w)- {X}_\epsilon(w)) \cdot \prod_{\substack{m=1\\m\neq w}}^{r}X(m).
\end{align*}	
We now set $Z(w)$ to be equal to the product appearing above.
Using the obvious upper bound $X(m)-{X}_\epsilon(m) \leq (X(m)-{X}_\epsilon(m))^{1/2} \cdot (X(m))^{1/2}$ and H\"older's inequality twice gives
\begin{align*}
	\E{R_n(\epsilon)}\leq \frac{1}{n^r} \cdot \sum_{w=1}^{r} \left( \E{X(w)-{X}_\epsilon(w)} \right)^{1/2} \cdot (\E{(X(m))^3})^{1/6} \cdot (\E{(Z(w))^{3}})^{1/3}.
\end{align*}
Lemma~\ref{lem:previouspaper} and~\eqref{eq:errortermxy} now give that 
\begin{align}\label{eq:rnsmall}
\E{R_n(\epsilon)}\leq C_1 \cdot \sqrt{\epsilon \log\left(\frac{1}{\epsilon} \right)} \to 0 \quad \text{ as } \epsilon\to 0
\end{align}
and this proves the first convergence.

For the second statement we note that using Cauchy-Schwarz and similar arguments as in Lemma~\ref{lem:gg}
one can remove the sets $D_\beta$ and $E_\beta$ and then apply the first part of the lemma.

Finally we get the convergence in expectation as a consequence of weak convergence together with uniform integrability which follows directly from~Lemma~\ref{lem:previouspaper}.
\end{proof}

\begin{proof}[\bf Proof of Proposition \ref{prop.moments.chi}]

 By Cramer--Wold in order to deduce the weak convergence it suffices to prove that all linear combinations of variables on the left converge weakly to the corresponding linear combinations of variables on the right.
Lemma \ref{lem:deloc} shows that one can replace the terms $\chi_n(i,j)$ by their localised versions, $\chi_{n,\alpha}(i,j)$.  Lemmas~\ref{lem:moments} and~\ref{lem:donsker}
show that the moments of all linear combinations of the $\chi_{n,\alpha}(i,j)$ do actually converge to the corresponding moments. We only need to ensure that the moments of the limiting object uniquely characterise the distribution. This now follows from Proposition~\ref{pro:carleman} using Carleman's criterion (see Section 3.3.5 in \cite{Durrett}). Indeed, for any variable $X$ on the right of \eqref{conv.law.cross}, Proposition~\ref{pro:carleman} gives that 
\[
\E{X^p}\leq C^p p^{2p}.
\]
Therefore, if $X$ and $Y$ are two of the variables on the right, then by the triangle inequality for the~$L^p$-norm we get
\[
\E{(X+Y)^p}^{1/p} \ \leq\  \E{X^p}^{1/p} + \E{Y^p}^{1/p} \ \leq\  2C\cdot  p^2.
\]
Therefore, Carleman's condition also holds for the sum $X+Y$, hence its distribution is uniquely characterised by its moments. 
\end{proof}


\section{Central limit theorem}\label{sec:clt}
In this section we finally give the proof of Theorem~\ref{thm:clt}.

We start by proving an upper bound on the variance of $\cc{\RR_n}$ using the same technique as Le Gall did for the range in dimension $2$. 

\begin{lemma}\label{lem:var}
We have $\E{(\cc{\RR_n}-\E{\cc{\RR_n}})^4}\lesssim n^4/(\log n)^{8}$.
\end{lemma}

\begin{proof}[\bf Proof]

The proof follows in the same way as~\cite[Lemma~6.2]{LG86} and~\cite[Lemma~3.5]{Asselah:2016hc} . We write $X_n=\cc{\RR_n}$ and $\overline{X}=X-\E{X}$ and we set for all $k\geq 1$
\[
a_k = \sup\left\{\sqrt{\E{\overline{X_n}^2}}: \, 2^k\leq n<2^{k+1} \right\}.
\]
For $k\geq 2$ we take $n$ such that $2^k\leq n<2^{k+1}$ and write $\ell=[n/2]$ and $m=n-\ell$. Then from Proposition~\ref{pro:cprange} we get
\[
|\overline{X_n} - \overline{X}_\ell^{(1)} - \overline{X}_m^{(2)}| = |\overline{\chi}_n(1,1) + \overline{\epsilon}_n|.
\]
Proposition~\ref{prop.moments.chi} and Lemma~\ref{lem:previouspaper} give that 
\[
\E{{\chi_n(1,1)}}\lesssim \frac{n}{(\log n)^2} \quad \text{ and }\quad 
\E{(\overline{\chi_n(1,1)})^2} \lesssim \frac{n^2}{(\log n)^4}.
\]
Also from Proposition~\ref{pro:cprange} we have that $\E{\epsilon_n}=\mathcal{O}(\log n)$ and $\E{\epsilon_n^2}=\mathcal{O}((\log n)^2)$. The proof for the fourth moment follows in exactly the same way as in~\cite[Lemma~4.2]{Asselah:2016hc} using the bound we first obtain on the variance.
\end{proof}

\begin{proof}[\bf Proof of Theorem~\ref{thm:clt}]

For any fixed $p\ge 1$, Proposition~\ref{pro:cprange} shows that 
\begin{align*}
	\cc{\RR_n} = \sum_{j=1}^{2^p}\cc{\RR_{n}^{(p,j)}} - \sum_{i=1}^{p}\sum_{j=1}^{2^{i-1}} \chi_n(i,j) + \epsilon_n.
\end{align*}
We write $\overline{X} = X-\E{X}$. 
Subtracting the expectation in the equation above we obtain
\begin{align*}
\overline{\cc{\RR_n}}	= \sum_{j=1}^{2^p}\overline{\cc{\RR_n^{(p,j)}}} - \sum_{i=1}^{p}\sum_{j=1}^{2^{i-1}} \overline{\chi_n(i,j)} + \overline{\epsilon_n}.
\end{align*}
Lemma~\ref{lem:var} and the independence of the ranges $\RR^{(p,j)}_n$ immediately give that 
\begin{align*}
	\E{\left(\frac{(\log n)^2}{n}\sum_{j=1}^{2^p}\overline{\cc{\RR_{n}^{p,j}}}\right)^2}\lesssim 2^{-p}.
\end{align*}
Since $\E{\epsilon_n} = o(n/(\log n)^2)$ from Lemma~\ref{lem:deloc} we get that \[
\frac{(\log n)^2}{n}\cdot{\overline{\epsilon_n}} \quad \stackrel{\bP}{\longrightarrow} \quad 0, \quad \text{ as } n\to\infty.
\]
Moreover, using Proposition \ref{prop.moments.chi} we get that for a fixed $p$ as $n\to\infty$
\begin{align*}
	\frac{2(\log n)^2}{\pi^4 \cdot n}\cdot \sum_{i=1}^{p}\sum_{j=1}^{2^{i-1}} \overline{\chi_n(i,j)}\quad 
\stackrel{(d)}{\Longrightarrow} \quad \sum_{i=1}^{p}\sum_{j=1}^{2^{i-1}} \left(\int_{A_{i,j}}G(\beta_s,\beta_t)\,ds\, dt - \E{\int_{A_{i,j}} G(\beta_s,\beta_t)\,ds\, dt} \right).
\end{align*}
From Proposition \ref{pro:conv} we also have in $L^2$-norm: 
\begin{align*}
	\lim_{p\to \infty }\ \sum_{i=1}^{p}\sum_{j=1}^{2^{i-1}}\left(\int_{A_{i,j}}G(\beta_s,\beta_t)\,ds\, dt - \E{\int_{A_{i,j}} G(\beta_s,\beta_t)\,ds\, dt} \right)\ =\  \gamma_G(\C_1)\ =\ \frac 12\, \gamma_G([0,1]^2). 
	\end{align*}
So first taking $p$ large enough and then letting $n\to \infty$ finishes the proof.
\end{proof}

\begin{proof}[\bf Proof of Corollary~\ref{cor:var}]

 Lemma~\ref{lem:var} shows that $(\log n)^4/n^2(\cc{\RR_n}- \E{\cc{\RR_n}})^2$ is uniformly integrable. Hence this together with the convergence in distribution from Theorem~\ref{thm:clt} proves the corollary.
\end{proof}

\section{Appendix}

In this section we prove the two lemma that were used in the proof of Proposition \ref{pro:carleman}. 

\begin{proof}[\bf Proof of Lemma \ref{momentLem1}] 
First using the explicit density of the Brownian motion we get 
\begin{eqnarray*}
\int_0^1 \mathbb E\left[\frac {(a+|\log \|\beta_s - x\||)^k}{\|\beta_s-x\|^2}\right]\, ds & \lesssim & 
\int_0^1 \,  \frac{ds}{s^2} \, \int_{\mathbb R^4}  \frac{(a+|\log \|u\| |)^k }{\|u\|^2}\,  e^{- \frac{\|u+x\|^2}{2s}} \, du\\
&\lesssim & \int_{\mathbb R^4}  \frac{(a+|\log \|u\| |)^k }{\|u\|^2\cdot \|u+x\|^2}\,  e^{- \frac{\|u+x\|^2}{2}} \, du,
\end{eqnarray*}
using Fubini at the second line. 
Now we cut the space in three regions defined as follows: 
\begin{equation*}
\mathcal U_1 : = \left\{\|u\|\ge 2 \|x\|\right\}, \quad \mathcal U_2:=\left\{\|u+x\| \ge \|x\|/2 \text{ and }\|u\|\le 2\|x\| \right\}, \quad 
\mathcal U_3 := \left\{\|u+x\| \le \|x\|/2 \right\}, 
\end{equation*}
and define 
$$F(u,x)\ := \  \frac{(a+|\log \|u\| |)^k }{\|u\|^2\cdot \|u+x\|^2}\cdot e^{- \frac{\|u+x\|^2}{2}}.$$
We bound this function on the three regions as follows 
\begin{eqnarray*}
F(u,x) \ \lesssim \ \left\{ 
\begin{array}{lc}
\frac{(a+|\log \|u\| |)^k }{\|u\|^4}\cdot e^{- \frac{\|u\|^2}{8}} & \text{ on } \mathcal U_1\\
\frac{(a+|\log \|u\| |)^k }{\|u\|^2\cdot \|x\|^2}\cdot e^{- \frac{\|x\|^2}{8}} & \text{ on } \mathcal U_2\\
\frac{(a+1+|\log \|x\| |)^k }{\|x\|^2\cdot \|u+x\|^2}\cdot e^{- \frac{\|u+x\|^2}{2}} & \text{ on } \mathcal U_3.\\
\end{array}
\right.
\end{eqnarray*}
Then we obtain with appropriate changes of variables 
$$\int_{\R^4} \, F(u,x)\, du \ \lesssim \ I_1(x) + I_2(x) + I_3(x),$$
with 
$$I_1(x) = \int_{2\|x\|}^\infty\,  \frac{(a+|\log r |)^k}{r} \cdot e^{- \frac {r^2}{8}} \, dr, $$
$$I_2 (x)= \|x\|^{-2}\cdot e^{- \frac{\|x\|^2}{8}}\,  \int_0^{2\|x\|} \, r \, (a+|\log r |)^k \, dr,$$
$$I_3 (x)= \frac{(a+1+|\log \|x\| |)^k }{\|x\|^2}  \int_0^{\|x\|} r\, e^{-r^2/2} \, dr .$$
Note that $I_3(x)\lesssim (a+1+|\log \|x\| |)^k$. Moreover, $I_1(x)$ and $I_2(x)$ will be bounded using the two following claims. 

\begin{claim}\label{cl:intlog}
For all $a\ge 0$, $b>0$ and $k\ge 1$ we have 
\[
\int_0^b (a+|\log r|)^k r\,dr\  \lesssim  \ b^2 \cdot \sum_{\ell=0}^{k}(a+|\log b|)^{k-\ell} \cdot k^\ell,
\]	
\end{claim}

\begin{proof}[\bf Proof]
Note 
$$f(k) := \int_0^b (a+|\log r|)^k r\,dr.$$
Using the change of variable $r=\exp(-u)$, we obatin  
$$f(k) = \int_{-\log b}^\infty (a+|u|)^k e^{-2u} \, du.$$
Assume first that $b<1$, so that $-\log b$ is nonnegative. Then an integration by parts gives 
$$f(k) = \frac {b^2}2 (a+|\log b|)^k  + \frac k2 f(k-1),$$
leading to the desired result by induction. Now if $b\ge 1$, one has 
$$f(k) = \int_0^{\log b} (a+u)^k e^{2u} \, du + \int_0^\infty (a+u)^k e^{-2u} \, du,$$
 and an integration by parts gives similarly 
 $$f(k) \, \lesssim \,    b^2(a+\log b)^k + k f(k-1),$$
 and the claim follows as well by induction. 
 \end{proof}

\begin{claim}
\label{claim.moment2}
For all $a\ge 0$, $b>0$, and $k\ge 0$ we have 	
\[
\int_b^\infty \frac{(a+|\log r|)^k}{r}\, e^{- r^2/8}\, dr\  \lesssim \ \frac{(a+|\log b|)^{k+1}}{k+1}+ e^{-b^2/8} \cdot \sum_{\ell =0}^{k} (4k)^\ell (a+|\log b|)^{k-\ell}.
\]
\end{claim}

\begin{proof}[\bf Proof]
Assume first that $b\ge 1$, and define   
$$g(k,b)= \int_b^\infty (a+|\log r|)^k\, r \, e^{- r^2/8}\, dr.$$
Note that 
$$\int_b^\infty \frac{(a+|\log r|)^k}{r}\, e^{- r^2/8}\, dr \, \le \, g(k,b).$$
Moreover, an integration by parts yields 
$$g(k,b) \, \le\,  4(a+\log b)^k e^{-b^2/8} + 4k g(k-1,b),$$
which gives the result by induction. Now if $b<1$, we have 
\begin{eqnarray*}
\int_b^\infty \frac{(a+|\log r|)^k}{r}\, e^{- r^2/8}\, dr &\le & \int_b^1 \frac{(a-\log r)^k}{r}\, dr + g(k,1)\\
&=& \frac{(a+|\log b|)^{k+1}}{k+1} + g(k,1),
\end{eqnarray*}
and using the previous estimate for $g(k,1)$, this concludes the proof of the claim. 
\end{proof}
Now we can just apply these two claims with $b=2\|x\|$ and use that $|\log 2 \|x\|| \le 1 + |\log \|x\||$. This gives the desired upper bounds for $I_1(x)$ and $I_2(x)$ and concludes the proof of Lemma \ref{momentLem1}.  
\end{proof}
\begin{proof}[\bf Proof of Lemma \ref{momentLem2}] We have 
\begin{eqnarray*}
\int_0^1 \mathbb E\left[ (a+|\log \| \widetilde \beta_t \| |)^k  \right] \, dt & \lesssim  & \int_0^1 \, \frac{dt}{t^2} \, \int_{\mathbb R^4} (a+|\log \|u \||)^k e^{- \|u\|^2/(2t)}\, du\\
& \lesssim &  \int_{\mathbb R^4} \frac{(a+|\log \|u \| |)^k}{\|u\|^2} e^{- \|u\|^2/2}\, du\\
& \lesssim & \int_0^\infty (a+|\log r|)^k \, r\, e^{- r^2/2}\, dr \\
&\lesssim &\int_0^1 (a+|\log r|)^k \, r\, dr + \int_1^\infty (a+|\log r|)^k \, r\, e^{- r^2/2}\, dr. 
\end{eqnarray*}
Now using the same argument as in the proof of Claim \ref{claim.moment2} for the second integral and Claim \ref{cl:intlog} with $b=1$ for the first one, we obtain the lemma.  
\end{proof}

\section*{Acknowledgements}

We thank Yinshan Chang for kindly sharing with us his manuscript,
which motivated us to look deeper at fluctuations in dimension four.

\bibliographystyle{abbrv}
\bibliography{biblio}

\end{document}